\documentclass[11pt]{amsart}
\usepackage{amsmath}
\usepackage{amssymb}
\usepackage{amsthm}
\usepackage{amscd}
\usepackage{fontenc}
\usepackage{url}
\usepackage{hyperref}
\usepackage[all]{xy}
\usepackage{graphicx}

\usepackage{tikz}
 \usetikzlibrary{arrows}
 
\numberwithin{equation}{section}

\newcounter{AbcT}

\newtheorem {Theorem}    {Theorem}[section]
\newtheorem* {Lemma9.3}    {Proposition 10.2}
\newtheorem* {Theoremmainprop}    {Theorem 1.3}

\newtheorem* {Question1.7}    {Question 1.7}

\newtheorem* {Definition} {Definition}
\newtheorem* {Remark}	{\bf{Remark}}

\newtheorem {Problem}    {Problem}
\newtheorem {Question}    {Question}

\newtheorem {Lemma}      [Theorem]    {Lemma}
\newtheorem {Corollary}   [Theorem] {Corollary}
\newtheorem {Proposition}[Theorem]    {Proposition}
\newtheorem {Claim}      [Theorem]    {Claim}
\newtheorem {Observation}[Theorem]    {Observation}

\theoremstyle{remark}

\newcounter{DM@bibnum}

\newcommand{\la}{\langle}
\newcommand{\ra}{\rangle}

\def\Nil{\mathbf{Nil}}
\def\Sol{\mathbf{Sol}}

\def\AGL{{\rm AGL}}
\def\SL{{\rm SL}}

\def\GL{{\rm GL}}
\def\PGL{{\rm PGL}}

\def\NCC{{\rm NCC}}
\def\CC{{\rm CC}}
\def\MC{{\rm MC}}
\def\PEO{{\rm PEO}}
\def\MEO{{\rm MEO}}

\def\log{{\rm log\,}}
\def\deg{{\rm deg\,}}

\def\Aut{{\rm Aut}}
\def\Inn{{\rm Inn}}
\def\Out{{\rm Out}}

\def\Ker{{\rm Ker\,}}
\def\Im{{\rm Im\,}}

\def\ord{{\rm ord}}

\def\rk{{\rm rk\,}}

\def\NSL_2{{\mathcal N SL_2}}

\def\sl{{\mathfrak{sl}}}
\def\gl{{\mathfrak{gl}}}

\def\Comm{{\rm Comm}}

\def\eps{\varepsilon}

\def\lam{\lambda}            
                
\def\phi{\varphi}

\def\calC{{\mathcal C}}

\def\calE{{\mathcal E}}
\def\calF{{\mathcal F}}

\def\hbar{\bar h}

\def\dbF{{\mathbb F}}

\def\dbN{{\mathbb N}}

\def\dbQ{{\mathbb Q}}

\def\dbZ{{\mathbb Z}}

\def\skv{{\vskip .1cm}}

\begin{document}
\title[Groups covered by conjugates of finitely many (pro)cyclic subgroups]
{On groups that can be covered by conjugates of finitely many cyclic or procyclic subgroups}

\author{Yiftach Barnea}
\address{Yiftach Barnea, Department of Mathematics, Royal Holloway, University of London, Egham, Surrey TW20 0EX, UK}
\email{y.barnea@rhul.ac.uk}
\author{Rachel Camina}
\address{Rachel Camina, Fitzwilliam College, Cambridge, CB3 0DG, UK}
\email{rdc26@cam.ac.uk}
\author{Mikhail Ershov}
\address{Mikhail Ershov, Department of Mathematics, University of Virginia, 141 Cabell Drive, Charlottesville, VA 22903, USA}
\email{ershov@virginia.edu}
\author{Mark L. Lewis}
\address{Mark L. Lewis, Department of Mathematical Sciences, Kent State University, Kent, OH 44242 USA}
\email{lewis@math.kent.edu}


\begin{abstract}
Given a discrete (resp. profinite) group $G$, we define $\NCC(G)$ to be the smallest number of cyclic (resp. procyclic) subgroups of $G$ whose conjugates cover $G$. In this paper we determine all residually finite discrete groups with finite NCC and give an almost complete characterization of profinite groups with finite NCC. 
\end{abstract}

\maketitle
\section{Introduction}
\label{sec:main}

\subsection{Motivation}

Questions of covering groups by conjugacy classes of subgroups, frequently called {\it normal coverings},  have a very long history. For instance, it is a classical theorem from the 19th century that a finite group cannot be written as a union of conjugates of a (single) proper subgroup.\footnote{This theorem is often attributed to Burnside and appears in his 1897 book \cite{Bu}.
However, an equivalent result stated in terms of permutation groups was already established by Jordan~\cite{Jo} in 1872.}
In modern terminology, this theorem asserts that finite groups are invariably generated, a property which attracted plenty of attention over the past decade (see, e.g., \cite{Mi} and references therein). 
A lot of recent work was also devoted to studying the {\it normal covering number} $\gamma(G)$ for a finite non-cyclic group $G$ -- the smallest number of proper subgroups whose conjugates cover $G$ (see, e.g., \cite{BPS} and references therein
as well as \cite{BSW} for the investigation of a related quantity $\gamma_w(G)$).

\skv

In this paper we will study {\it normal cyclic coverings}, that is, coverings of groups by conjugacy classes of {\it cyclic} subgroups. The main invariant we will be interested in
is defined as follows.

\begin{Definition}\rm Let $G$ be a group. We define $\NCC(G)$ to be the smallest $k$ such that $G$ can be written as a union of conjugacy classes of $k$ cyclic subgroups. If no such $k$ exists, we set $\NCC(G)=\infty$. 
\end{Definition}

Our motivation for studying NCC was two-fold. On one hand, understanding which infinite groups have finite NCC and the closely related property (BVC) is related to certain problems about classifying spaces for families of subgroups, most notably a conjecture of Juan-Pineda and Leary \cite[Conjecture~1]{JPL} and a question of L\"uck, Reich, Rognes and Varisco~\cite[Question~4.9]{LRRV} (see \S~\ref{sec:BVC} for details). On the other hand, it is natural to compare $\NCC(G)$ with the classical and much better understood invariant $k(G)$, the number of conjugacy classes of $G$. One of the basic properties of $k(G)$ is that for finite $G$, it grows with the size of the group: $k(G)\to \infty$ if $|G|\to \infty$. Thus one may ask the following question:

\begin{Question} 
\label{q:NCCgrows}
Let $\calC$ be a class of finite groups. Is it true that $\NCC(G)\to\infty$ as $|G|\to\infty$ for $G\in\calC$?
\end{Question}

The answer to Question~\ref{q:NCCgrows} is clearly negative if $\calC$ contains all finite groups since $\NCC(G)=1$ for any cyclic group. Excluding cyclic groups is not sufficient for a positive answer as it is easy to see that all non-abelian groups of order $pq$, with $p$ and $q$ distinct primes, have NCC equal to $2$. Von Puttkamer asked in his Ph.D. thesis whether the answer is positive if $\calC$ is the class of all non-cyclic finite $p$-groups for a fixed $p>2$ \cite[Question~5.0.9]{vP}, and this question served as the original motivation for this project. 

It is natural to approach von Puttkamer's question via pro-$p$ groups. If $G$ is a profinite group, $\NCC(G)$ is defined in the same way as for discrete\footnote{In this paper by a discrete group we will simply mean a group not endowed with any topology.} groups except that one replaces cyclic subgroups by procyclic subgroups (that is, closed subgroups topologically generated by a single element). 
A standard argument (see Claim~\ref{claim:finitefamily}) 
shows that if for some $k\in\dbN$ there exist infinitely many (non-isomorphic) non-cyclic finite $p$-groups $G$ with $\NCC(G)\leq k$, then there exists an infinite non-procyclic
pro-$p$ group $G$ with $\NCC(G)\leq k$. Conversely, it is clear that if $G$ is any infinite pro-$p$ group which is not procyclic and $\NCC(G)=k < \infty$, then sufficiently large finite quotients of $G$ form an infinite family of non-cyclic finite $p$-groups with NCC equal to $k$.

This led us to investigate which infinite pro-$p$ groups have finite NCC. As we will explain below, infinite non-procyclic pro-$p$ groups with finite NCC do exist, and thus von Puttkamer's question has negative answer. However, it turns out that infinite pro-$p$ groups and more generally infinite profinite groups with finite NCC have very restricted structure (see Theorems~\ref{main:prop}~and~\ref{thm:mainprofinite} and Proposition~\ref{main:pronilpotent}). Using these results, we will solve the aforementioned conjecture from \cite{JPL} and give a positive answer to \cite[Question~4.9]{LRRV} for discrete {\it residually finite} groups
(see Corollary~\ref{cor:top}). Going back to von Puttkamer's question, the proof of Claim~\ref{claim:finitefamily} shows that Theorem~\ref{main:prop} also yields strong constraints on families of finite $p$-groups with bounded NCC and can possibly provide the first step towards a satisfactory description of all such families. We are planning to address the latter problem in a follow-up paper.

\begin{Remark}\rm
We would like to mention a simple characterization of NCC valid for all profinite groups (so in particular for finite groups).
If $G$ is profinite, then $\NCC(G)$ is the number of conjugacy classes of maximal procyclic subgroups of $G$. This is because in a profinite group every procyclic subgroup is contained in a maximal procyclic subgroup. The corresponding assertion in the discrete case (with procyclic subgroups replaced by cyclic subgroups) does not always hold, even for residually finite groups. For example, $G=\oplus_{p}\dbZ/p\dbZ$, where the sum is over all primes, is a residually finite group which has infinite NCC but has no maximal cyclic subgroups. 
\end{Remark}

\subsection{Discrete groups with finite NCC}
Our first main theorem asserts that in the discrete residually finite case there are no non-trivial examples with finite NCC,
confirming a conjecture of von Puttkamer~\cite[Conjecture~5.0.1]{vP}:

\begin{Theorem}
\label{main:discrete}
Let $G$ be an infinite discrete residually finite group with finite NCC. Then $G$ is infinite cyclic or infinite dihedral (both of these do have finite NCC, 1 and 3 respectively).
\end{Theorem}

There are several classes of infinite discrete (not necessarily residually finite) groups which were previously known to satisfy the 
implication of Theorem~\ref{main:discrete}:
\begin{itemize}
\item[(a)] virtually solvable groups,
\item[(b)] one-relator groups,
\item[(c)] acylindrically hyperbolic groups,
\item[(d)] 3-manifold groups,
\item[(e)] CAT(0) cube groups,
\item[(f)] finitely generated linear groups,
\item[(g)] arbitrary linear groups in characteristic zero. 
\end{itemize}

In other words, every infinite group with finite NCC in one of these classes is either infinite cyclic or infinite dihedral.
For (a) this was proved by Groves and Wilson~\cite{GW}. Von Puttkamer and Wu proved the result for classes (b)-(e) in \cite{vPW} and for (f) in \cite[Theorem~2.11]{vPW2}.\footnote{Technically, the results for all classes (a)-(f) were not established until \cite{vPW2} since \cite{GW} and \cite{vPW} dealt not with groups with finite NCC, but with groups satisfying the related property (BVC) --
see \S~\ref{sec:BVC}. However, the proofs of the corresponding results for (BVC) are completely analogous.} 
Finally, (g) is a combination of (a) and a theorem of Bernik~\cite{Be} (see Theorem~\ref{Bernik} for the statement) which, in turn, is based on the existence
of generic elements in Zariski-dense subgroups of semisimple algebraic groups in characteristic zero, established by
Prasad and Rapinchuk in \cite{PR} (see also Proposition~3.5 and Remark~3.6 in \cite{CRRZ}).

Since finitely generated linear groups are residually finite, the result for (f) is a special case of Theorem~\ref{main:discrete}. However, 
we originally proved Theorem~\ref{main:discrete} only for finitely generated groups, and the proof relied on the corresponding result for (f). To prove Theorem~\ref{main:discrete} in the general case we will use a similar strategy, but
instead of \cite[Theorem~2.11]{vPW2} we will apply the above theorem from \cite{Be}.

\vskip .12cm
Note that if a group $G$ has finite NCC, then obviously so do all its quotients. Thus, we get an immediate consequence of Theorem~\ref{main:discrete} applicable to arbitrary discrete groups.

\begin{Corollary}
\label{cor:maindiscrete} Let $G$ be a discrete group with finite NCC. Then the largest residually finite quotient
of $G$ (which is the image of $G$ in its profinite completion) is finite, cyclic or infinite dihedral. 
\end{Corollary}

\begin{Remark}\rm There are plenty of known examples of infinite discrete groups which have finitely many conjugacy classes
and thus in particular have finite NCC. Such groups with only 2 conjugacy classes (albeit infinitely generated) were constructed already in the classical paper of Higman, B.H. Neumann and H. Neumann~\cite{HNN}. To the best of our knowledge, the first finitely generated examples are due to S. Ivanov~\cite[Theorem~41.2]{Ol} who in particular showed that there exist such groups of exponent $p$ for every sufficiently large prime $p$. Finally, the main theorem of a remarkable paper of Osin~\cite{Os} implies that for any $n\geq 2$ there exist infinite $2$-generated groups with exactly $n$ conjugacy classes. For additional examples of infinite groups with finite NCC see \cite{vPW2}.   
\end{Remark}

\subsection{Profinite groups with finite NCC}

We now turn to the classification of profinite groups with finite NCC.  We start by describing pro-$p$ groups with finite NCC.

\begin{Theorem}
\label{main:prop}
Let $p$ be a prime and $G$ a pro-$p$ group. Then $G$ has finite NCC if and only
if one of the following 3 mutually exclusive conditions holds:
\begin{itemize}
\item[(i)] $G$ is finite. 
\item[(ii)] $G$ is infinite procyclic or $p=2$ and $G$ is infinite prodihedral, that is, the pro-$2$ completion of the infinite dihedral group.
\item[(iii)] $G$ is isomorphic to an open torsion-free subgroup of $\PGL_1(D)$ where $D$ is the quaternion division algebra\footnote{Such a division algebra is unique (up to isomorphism). Indeed, for any field $F$ the number of isomorphism classes of central division of degree $d$ is equal to the number of elements of order $d$ in the Brauer group $Br(F)$. It is well known that the Brauer group
of any non-archimedean local field is isomorphic to $\dbQ/\dbZ$ and thus has a unique element of order $2$.}
over $\dbQ_p$. 
\end{itemize}
\end{Theorem}
\begin{Remark}\rm Let us briefly comment on the structure of the groups in item (iii). Let $D$ be the quaternion division algebra over $\dbQ_p$
and $O_D$ its ring of integers. The group $\PGL_1(D)=D^{\times}/\dbQ_p^{\times}$ is virtually pro-$p$ and virtually torsion-free. Moreover its first congruence subgroup $\PGL_1^1(O_D)$ is pro-$p$ and for $p>2$ contains every pro-$p$ subgroup of $\PGL_1(D)$.
It is easy to show that if $p>3$, already the group $\PGL_1^1(O_D)$ is torsion-free.
\end{Remark}

Let $\Nil$ denote the class of finite nilpotent groups. The classification of pro-$\Nil$ groups with finite NCC easily reduces to the pro-$p$ case. Indeed, if $G$ is pro-$\Nil$, it is a direct product of its Sylow pro-$p$ subgroups $G_p$. Moreover, by Lemma~\ref{lem:product} below we have $\NCC(G)=\prod \NCC(G_p)$. Thus a pro-$\Nil$ group $G$ has finite NCC
if and only if each $G_p$ has finite NCC and moreover $\NCC(G_p)=1$ for almost all $p$. Since pro-$p$ groups with NCC 1 are exactly procyclic pro-$p$ groups and a product of procyclic groups of coprime orders is procyclic, we conclude the following:

\begin{Proposition}
\label{main:pronilpotent}A pro-$\Nil$ group $G$ has finite NCC if and only if
$G=C\times \prod_{i=1}^k H_i$ where $C$ is a procyclic group and there exist distinct primes $p_1,\ldots, p_k$ not dividing $|C|$
such that each $H_i$ is a non-cyclic pro-$p_i$ group with finite NCC.
\end{Proposition}

Our last main theorem deals with arbitrary profinite groups with finite NCC.

\begin{Theorem}
\label{thm:mainprofinite}
Let $G$ be a profinite group with finite NCC. Then $G$ contains an open pro-$\Nil$ subgroup (which must also have finite NCC by Lemma~\ref{lem:finiteindex}).
\end{Theorem}

Note that Theorem~\ref{main:prop}, Proposition~\ref{main:pronilpotent} and Theorem~\ref{thm:mainprofinite} completely characterize profinite group which have an open subgroup with finite NCC. However, they do not provide a classification of profinite groups with finite NCC up to isomorphism since finiteness of NCC is not necessarily preserved by passing to finite index overgroups. 

\skv

{\bf Profinite groups with countable NCC.} Recently Jaikin-Zapirain and Nikolov~\cite{JN} proved that any infinite compact Hausdorff group (in particular, any infinite profinite group) has uncountably many conjugacy classes 
(see also \cite{Wil1}~and~\cite{Wil2} for some more refined results of this type).
Several recent papers investigated profinite groups in which a countable union of procyclic subgroups (without taking conjugates) contains a large portion of the group (in a suitable sense) -- see, e.g. \cite{AS}.
 
As a natural continuation of this line of research with our results, we propose the following problem.

\begin{Problem} 
\label{prob:countableNCC}
Classify profinite groups with {\bf countable} NCC.
\end{Problem}

A simple example of a profinite group with countable, but infinite NCC is given by $\dbZ_p\times \dbZ/p\dbZ$. More generally, it is easy to show
that every virtually procyclic group has countably many maximal procyclic subgroups, and therefore every profinite group with (BVC) has countable NCC.
There do exist groups with countable NCC and without (BVC), e.g. $\dbZ_p^{\times}\ltimes \dbZ_p$ (the group of affine transformations of $\dbZ_p$) and $\PGL_2(\dbZ_p)$. One can check that $\dbZ_p^{\times}\ltimes \dbZ_p$ has countable NCC directly from definition. The latter combined with the proof of \cite[Theorem~G]{BJL}
implies that $\PGL_2(\dbZ_p)$ has countable NCC. Despite these additional examples, it is feasible that the class of infinite profinite groups with countable NCC is still quite small. 

A standard argument using Baire Category Theorem shows that a profinite group $G$ with countable NCC must have a procyclic subgroup $C$ such that
$\cup_{g\in G}C^g$ has non-empty interior. Thus, as a further generalization of Problem~\ref{prob:countableNCC} one can ask what can be said about the groups with the latter property. We are grateful to Colin Reid for proposing this question.
We refer the reader to \cite{Wes} for a discussion of the corresponding problem about conjugacy classes of elements (classify profinite groups which have a conjugacy class with non-empty interior); see also
\cite[Question~2]{JN}.

\subsection{Outline of the paper} $\empty$ 
\label{sec:outline}

\begin{itemize}
\item In \S~2, we introduce a certain generalization of the NCC invariant, $\CC(G,\Phi)$, where $\Phi$ is a group acting on $G$ by automorphisms, and prove some general results about it.
\item The proof of Theorem~\ref{thm:mainprofinite} is divided into three parts, which will be established in \S~3,~4~and~7, respectively.
\begin{itemize}
\item In \S~3 we prove that a profinite group with finite NCC has an open pro-$\Sol$ subgroup (which also has finite NCC by Lemma~\ref{lem:finiteindex}). Here $\Sol$ denotes the class of finite solvable groups.
\item In \S~4 we prove that if $G$ is a pro-$\Sol$ group with finite NCC, then  for some $k\in\dbN$ the $k^{\rm th}$ term of its derived series $G^{(k)}$
is pro-$\Nil$.
\item Finally, in \S~7 we prove that if $G$ is a pro-$\Sol$ group with finite NCC such that $G^{(k)}$
is pro-$\Nil$ for some $k\in\dbN$, then $G$ is virtually pro-$\Nil$.
\end{itemize}
\item In \S~5 we will prove Theorem~\ref{main:discrete} assuming Theorem~\ref{thm:mainprofinite} (whose proof will be completed in \S~7)
and Theorem~\ref{main:prop}.
\item In \S~6 we will prove that pro-$p$ groups with finite NCC are $p$-adic analytic and then use this result to prove
Theorem~\ref{main:prop}. 
\item Finally, in \S~8 we will introduce property (BVC), a certain variation of finiteness of NCC, and explain 
why Theorem~\ref{main:discrete} settles certain questions in topology dealing with classifying spaces for families of subgroups.
\end{itemize}

\paragraph {\bf Acknowledgments. }\rm We are indebted to the anonymous referee for extremely helpful and detailed feedback on an earlier version of the paper, which resulted in major improvements in the exposition. In addition, the proofs of the following results in the present version were either explicitly proposed by the referee or heavily rely on the suggestions from the report:  Lemma~\ref{cor:proptorsion}, 
Lemma~\ref{lem:NCCrest}, Lemma~\ref{lem:Z}, Lemma~\ref{lem:ji2}, Lemma~\ref{lem:dp}, Proposition~\ref{reduction:ressolvable0}, Theorem~\ref{main:prop} and Proposition~\ref{lem:metabelian}.

We are very grateful to Xiaolei Wu for asking\footnote{Xiaolei Wu asked this question at a `Functor Categories for Groups' meeting, which was supported by a London Mathematical Society Joint Research Group grant.} 
us a version of von Puttkamer's question~\cite[Question~5.0.9]{vP}. We would like to thank Andrei Rapinchuk for illuminating discussions and suggesting the reference \cite{Be} and
Alex Lubotzky for bringing \cite{BJL} to our attention. We would also like to thank Andrei Jaikin-Zapirain, Ian Leary, Alex Lubotzky, Colin Reid 
and John Wilson for helpful feedback on earlier versions of this paper.

\section{Cyclic covering number relative to a group of automorphisms}

\subsection{Covering numbers for subgroups, quotients and direct products}

While we are primarily interested in NCC, in the proofs it will be very convenient to work with a certain generalization of NCC defined below which has better hereditary properties.

\begin{Definition}\rm Let $G$ be a discrete (resp. profinite) group and $\Phi$ a group acting on $G$ 
by group automorphisms\footnote{By a homomorphism between profinite groups we will always mean a continuous homomorphism unless explicitly indicated otherwise.}. A {\it cyclic (resp. procyclic) $\Phi$-cover} of $G$ is a collection of cyclic (resp. procyclic) subgroups $\{C_i\}_{i\in I}$ of $G$
such that $G=\bigcup\limits_{i\in I,\phi\in \Phi}\phi(C_i)$. We define $\CC(G,\Phi)$ to be the smallest number of subgroups in a cyclic (resp. procyclic)
$\Phi$-cover of $G$.
\end{Definition}

Note that $\NCC(G)=\CC(G,G)$ (where $G$ acts on itself by conjugation).

\begin{Lemma}
\label{CC:subgroup}
Let $G$ be a group and $\Phi$ a group acting on $G$ by automorphisms. The following hold:
\begin{itemize}
\item[(i)]If $H$ is a $\Phi$-invariant subgroup of $G$, then $\CC(H,\Phi)\leq \CC(G,\Phi)$.
In particular, if $H$ is any normal subgroup of $G$, then $\CC(H,G)\leq \NCC(G)$.
\item[(ii)] If $K$ is a $\Phi$-invariant normal subgroup of $G$ (so that $\Phi$ naturally acts on $G/K$), then
$\CC(G/K,\Phi)\leq \CC(G,\Phi)$. In particular, if $K$ is any normal subgroup of $G$, then $\NCC(G/K)\leq \NCC(G)$.
\item[(iii)] If $\Psi$ is a finite index subgroup of $\Phi$, then $\CC(G,\Psi)\leq [\Phi:\Psi]\CC(G,\Phi)$. 
\end{itemize}
\end{Lemma}
\begin{proof} (i) and (ii) are obvious, and (iii) follows from the fact that for any action of $\Phi$ on a set, any orbit of $\Phi$ is a union of at most $[\Phi:\Psi]$ orbits of $\Psi$.
\end{proof}

The next result which follows from Lemma~\ref{CC:subgroup} and has been well known before is particularly useful.

\begin{Lemma}
\label{lem:finiteindex}
Let $G$ be a group with finite NCC and $H$ a subgroup of finite index. Then $H$ also has finite NCC and in fact $\NCC(H)\leq [G:H]\cdot \NCC(G)$
\end{Lemma}
\begin{proof}
We have $\NCC(H)=\CC(H,H)\leq [G:H]\cdot \CC(H,G)\leq [G:H]\cdot \NCC(G)$ where both $\CC$ numbers are with respect to the conjugation action, the first inequality holds by
Lemma~\ref{CC:subgroup}(iii) and the second inequality holds by Lemma~\ref{CC:subgroup}(i). 
\end{proof}

\begin{Definition}\rm 
Let $G$ be a profinite group. The {\it order of $G$} is the supernatural number defined as the least common multiple of the orders of finite quotients of
$G$ (a supernatural number is a formal product $\prod_{p}p^{\alpha_p}$ where $p$ ranges over all primes and each 
$\alpha_p\in\dbZ_{\geq 0}\cup\{\infty\}$).
\end{Definition}

\begin{Lemma}
\label{lem:product}
Let $G$ and $H$ be discrete or profinite groups and let $\Phi$ and $\Psi$ be groups acting by automorphisms on $G$ and $H$, respectively,
such that $\CC(G,\Phi)$ and $\CC(H,\Psi)$ are both finite. The following hold:
\begin{itemize}
\item[(a)] $\CC(G\times H, \Phi\times \Psi)\geq \CC(G,\Phi)\cdot \CC(H,\Psi)$. In particular, if $G$ and $H$ both have finite NCC,
$$\NCC(G\times H)\geq \NCC(G)\cdot \NCC(H).$$ 
\item[(b)] Assume now that $G$ and $H$ are profinite and have coprime orders. Then both inequalities in (a) must be equalities.
\end{itemize}

\end{Lemma}
\begin{proof} (a) For simplicity we will present a proof in the discrete case. The argument in the profinite case is completely analogous.
Let $n=\CC(G,\Phi)$, $m=\CC(H,\Psi)$ and $t=\CC(G\times H, \Phi\times \Psi)$, and assume that $t$ is finite (if $t$ is infinite, there is nothing to prove).

Let $\{ C_k \}_{k=1}^t$ be a cyclic $\Phi\times \Psi$-cover of $G\times H$.
Let $G_k$ and $H_k$ denote the projections of $C_k$ to $G$ and $H$, respectively. Obviously, $G_k$ and $H_k$ are cyclic and $(\phi \times \psi)(C_k) \subseteq \phi(G_k) \times \psi(H_k)$ for all $\phi\in\Phi$ and $\psi\in\Psi$. (Notice that $\phi(G_k) \times \psi(H_k)$ need
not be cyclic). Thus, $\{ G_k \times H_k \}_{k=1}^t$ is a $\Phi \times \Psi$-cover of $G \times H$, that is, 
$$G \times H=\bigcup\limits_{1 \leq k \leq t,\phi\in \Phi, \psi\in \Psi} (\phi \times \psi) (G_k \times H_k)=\bigcup\limits_{1 \leq k \leq t,\phi\in \Phi, \psi\in \Psi} \phi(G_k) \times\psi(H_k). \eqno (***)$$ 

If $G_i \subseteq \phi(G_{j})$ for some $i \neq j$ and $\phi\in \Phi$, we can replace $G_i$ by $G_{j}$ and still have a $\Phi\times \Psi$-cover of $G\times H$. After applying this operation finitely many times, we obtain a new cover which can be written as $\{ G_k \times H_{k,j} \}_{1 \leq k \leq n', 1 \leq j \leq m_k}$ where for $i\neq j$ we have $G_i \not\subseteq \phi(G_{j})$ for any $\phi\in \Phi$. By construction the number of sets in the new cover does not exceed the number of sets in the original cover, that is, $m_1+m_2+\cdots+m_{n'}\leq t$. Therefore, it suffices to show that $n' \geq n$ and $m_k \geq m$ for each $k$. 

Projecting both sides of (***) to the first component, we see that $\{ G_k \}_{k=1}^{n'}$ is a cyclic $\Phi$-cover of $G$, so $n' \geq n$. We now need to show that for a fixed $k$ the collection $\{ H_{k,j} \}_{j=1}^{m_k}$ is a cyclic $\Psi$-cover of $H$. Let $x_k$ be a generator of $G_k$. If $i \neq j$, then 
$x_i \not\in \phi(G_{j})$ for any $\phi\in \Phi$ (for otherwise, $G_i \subseteq \phi(G_{j})$ contrary to our assumption). Hence for any $h\in H$, the pair $(x_k,h)$ must be in $\phi(G_k) \times \psi(H_{k,j})$ for some $1 \leq j \leq m_k$, $\phi \in \Phi$ and $\psi \in \Psi$ , so 
$h \in\psi(H_{k,j})$. We conclude that $\{ H_{k,j} \}_{j=1}^{m_k}$ is a cyclic $\Psi$-cover of $H$ and thus $m_k \geq m$ as desired.

\skv
(b) Let $\{ G_i \}_{i=1}^n$ be a cyclic $\Phi$-cover of $G$ and $\{ H_j \}_{j=1}^m$ be a cyclic $\Psi$-cover of $H$. Since $G$ and $H$ have coprime orders, this is true also for every $G_i$ and $H_j$ and therefore $G_i \times H_j$ is procyclic (by the Chinese Remainder Theorem). Hence
$\{ G_i \times H_j \}_{1 \leq i \leq n, 1 \leq j \leq m}$ is a cyclic $\Phi\times \Psi$-cover of $G\times H$. Thus $\CC(G\times H, \Phi\times \Psi)\leq  nm= \CC(G,\Phi)\cdot \CC(H,\Psi)$,
and by (a) the equality must hold.
\end{proof}

\subsection{Some restrictions on groups with finite NCC}

In this subsection we will establish several results which impose restrictions on the structure of discrete and profinite groups with finite NCC. 

The following lemma proposed to us by the referee yields a strong restriction on the
torsion in pro-$p$ groups with finite NCC:

\begin{Lemma}
\label{cor:proptorsion}
Let $G$ be a pro-$p$ group with finite NCC and $T$ the set of its torsion elements. Then the set $T\setminus\{1\}$ is open.
In particular, either $T=\{1\}$ or $T$ has non-empty interior.
\end{Lemma}
\begin{proof} If $G=T$, there is nothing to prove, so assume that $G\neq T$.

By assumption there exist finitely many elements $x_1,\ldots, x_k$ such that 
$G=\bigcup\limits_{i=1}^k\overline{\la x_i\ra}^G$. 
Since $G$ is pro-$p$, any procyclic subgroup of $G$
is either finite or torsion-free. Thus, if $Z=\bigcup\limits_{x_i\not\in\, T}\overline{\la x_i\ra}^G$,
then $Z=(G\setminus T)\cup \{1\}=G\setminus (T\setminus\{1\})$ (note that $Z$ is non-empty and in particular contains $1$
since $G\neq T$).
\skv
Each set $\overline{\la x_i\ra}^G$ is compact (and hence closed in $G$) as it is a continuous image of the compact
topological space $\overline{\la x_i\ra}\times G$. Hence $Z$ is closed in $G$ and therefore
$T\setminus\{1\}=G\setminus Z$ is open in $G$.
\end{proof}

We will explicitly use Lemma~\ref{cor:proptorsion} in the proof of Theorem~\ref{main:prop} in \S~6; however, we will also
need the following generalization, both in \S~6 and later in this subsection: 

\begin{Lemma}
\label{lem:NCCrest} Let $G$ be a profinite group with finite NCC. Suppose that $G$ can be written as a disjoint union $G=A\sqcup B\sqcup C$
such that the following conditions hold:
\begin{itemize}
\item[(1)] $A, B$ and $C$ are normals subsets (that is, invariant under conjugation);
\item[(2)] $\overline{\la x\ra}\cap A=\emptyset$ for all $x\in B\sqcup C$ (where $\overline{\la x\ra}$ is the closure of $\la x\ra$);
\item[(3)] $\overline{\la a\ra}\cap C=\emptyset$ for all $a\in A$.
\end{itemize}
Then $G$ has an open subset $U$ such that $A\subseteq U\subseteq A\sqcup B$.
\end{Lemma}
\begin{Remark}\rm The last assertion of Lemma~\ref{cor:proptorsion} follows from Lemma~\ref{lem:NCCrest} applied
with $A=T\setminus\{1\}$, $B=\{1\}$ and $C=G\setminus T$, where $T$ is the set of torsion elements.
\end{Remark}
\begin{proof} As in the proof of Lemma~\ref{cor:proptorsion}, there exist finitely many elements $x_1,\ldots, x_k$ such that 
$G=\bigcup\limits_{i=1}^k\overline{\la x_i\ra}^G$. Let $I=\{i: x_i\in A\}$,
$Z=\bigcup\limits_{i\not\in I}\overline{\la x_i\ra}^G$ and $U=G\setminus Z$.
\skv
As in the proof of Lemma~\ref{cor:proptorsion}, $Z$ is closed whence $U$ is open.
Conditions (1) and (2) imply that $A\cap Z=\emptyset$, so $A$ is contained in $U$. On the other hand,
(1) and (3) imply that $\bigcup\limits_{i\in I}\overline{\la x_i\ra}^G\cap C=\emptyset$, so
$C\subseteq Z$ and therefore $U=G\setminus Z\subseteq G\setminus C=A\sqcup B$.
\end{proof}

The remaining results in this subsection deal with quotients of groups with finite NCC. We start with the technically easier discrete case.

\begin{Lemma}
\label{lem:Z}
 Let $G$ be a residually finite discrete group with finite NCC. Then either $G$ is infinite cyclic or $G$ has finite abelianization.
\end{Lemma}

\begin{proof} The abelianization $G/[G,G]$ is an abelian group with finite NCC, so it is a union of finitely many cyclic subgroups
and in particular finitely generated. Thus either $G/[G,G]$ is finite or $G/[G,G]$ maps onto $\dbZ$ (whence so does $G$). In the latter
case we can write $G=H\rtimes \dbZ$ for some normal subgroup $H$, and it remains to show that $H=\{1\}$.

Suppose that $H\neq \{1\}$. Since $G$ is residually finite, it has a finite index normal subgroup
$U$ which does not contain $H$. But then $G/(U\cap H)\cong F\rtimes \dbZ$ for some non-trivial finite group $F$ (isomorphic to $H/U\cap H$),
and by \cite[Lemma~3.7]{vPW2} any semidirect product of this form has infinite NCC, a contradiction.
\end{proof}

In the profinite case we will prove a somewhat similar, but more technical result:

\begin{Lemma}
\label{lem:ji2} 
Let $G$ be a profinite group, $H$ a (closed) normal subgroup of $G$ and $Q=G/H$. Suppose that $Q$ has a pro-$p$ element $x$ of infinite
order (that is, $\overline{\la x\ra}\cong \dbZ_p$) and $|H|$ is divisible by $p$. Then $G$ has infinite NCC. 
\end{Lemma}
\begin{proof} We first consider the special case where $G=H\times Q$ and $|H|=p$. 
Given an element $g\in G$, let $ord_p(g)\in\dbZ_{\geq 0}\cup\{\infty\}$
denote the exponent of $p$ in the order of $g$ (considered as a supernatural number). Equivalently, $p^{ord_p(g)}$
is the order of the Sylow pro-$p$ subgroup of $\overline{\la g\ra}$.

Suppose now that $G$ has finite NCC. Let $A=\{g\in G: 0<ord_p(g)<\infty\}$, $B=\{g\in G: ord_p(g)=0\}$ and $C=\{g\in G: ord_p(g)=\infty\}$. Clearly,
$G=A\sqcup B\sqcup C$, and it is straightforward to check that the hypotheses of Lemma~\ref{lem:NCCrest} hold.
Thus, $G$ has an open subset $U$ with $A\subseteq U\subseteq A\sqcup B$.

Consider the pro-$p$ subgroup $P=H\times \overline{\la x\ra}\cong \dbZ/p\dbZ\times \dbZ_p$.
Then $A\cap P=H\setminus\{1\}$ and $B\cap P=\{1\}$, so $U\cap P$ is an open subset of $P$ containing
$H\setminus\{1\}$ and contained in $H$. Thus, $P$ contains a non-trivial finite open set, a contradiction.
\skv

We now treat the general case. Since $|H|$ is divisible by $p$ and the topology on $H$ is induced from $G$, there exists an open normal subgroup $U$ of $G$ such that
the image of $H$ in $G/U$ has order divisible by $p$. Let $\pi_U:G\to G/U$ and $\pi_H:G\to G/H=Q$ be the natural
projections, and consider the map $\pi:G\to G/U\times Q$ given by $\pi(g)=(\pi_U(g),\pi_H(g))$. Then 
$\pi(UH)$ is an open subgroup of $\pi(G)$, and note that $\pi(UH)=\pi_U(H)\times \pi_H(U)$. By construction $\pi_U(H)$ has
an element of order $p$, call it $g$, and $\pi_H(U)$ has a pro-$p$ element of infinite order (being an open subgroup of $Q=\pi_H(G)$).
Hence $\la g\ra\times \pi_H(U)$ is an open subgroup of $\pi(G)$ which satisfies the hypotheses of the special case considered
at the beginning of the proof and thus has infinite NCC. Since finiteness of NCC is inherited by open subgroups and homomorphic images, it follows that $G$ also has infinite NCC, as desired.
\end{proof}

The following corollary of Lemma~\ref{lem:ji2} yields a much stronger conclusion in the pro-$p$ case.

\begin{Corollary}
\label{lem:ji}
Let $G$ be an infinite pro-$p$ group with finite NCC. Then $G$ must be just-infinite (that is, all of its proper continuous quotients are finite).
\end{Corollary}
\begin{proof} First by Lemma~\ref{propbound:NCC} below, $G$ is finitely generated. Suppose that $G$ has a non-trivial closed normal subgroup $H$ such that $G/H$ is infinite. By the positive solution to the general Burnside problem for pro-$p$ groups \cite{Ze}, a finitely generated torsion pro-$p$ group is finite. Thus, $G/H$ must contain an element $x$ of infinite order. Then $x$ and $H$ trivially satisfy the hypotheses
of Lemma~\ref{lem:ji2} and hence $G$ has infinite NCC, contrary to our assumption.
\end{proof}

\subsection{Other lower bounds on NCC}
In this subsection we collect some additional results which provide either a lower bound on NCC of a group or a restriction on the structure of a group
with finite NCC.
\skv

We start by bounding the NCC of a pro-$p$ group in terms of its number of generators.
\begin{Lemma}
\label{propbound:NCC} Let $G$ be a pro-$p$ group and $d(G)$ its minimal number of generators. Then $NCC(G)\geq \frac{p^{d(G)-1}}{p-1}$.
In particular, if $G$ has finite NCC, then $G$ is finitely generated.
\end{Lemma}
\begin{proof} Let $\Phi(G)$ denote the Frattini subgroup of $G$. By \cite[Propostions~1.9~and~1.13]{DDMS}, $G/\Phi(G)$
is an elementary abelian $p$-group with $d(G)=d(G/\Phi(G))$. Hence 
$$\NCC(G)\geq \NCC(G/\Phi(G))=\frac{|G/\Phi(G)|-1}{p-1}=\frac{p^{d(G)}-1}{p-1}.\qquad\qedhere$$
\end{proof}

Next we relate NCC to the set of orders of elements. The following definition will only be introduced for discrete groups. The corresponding notion in the profinite case requires extra care, but also will not be needed; in fact, here the case of finite groups will be sufficient for our purposes.

\begin{Definition}\rm
Let $G$ be a non-trivial discrete group.
\begin{itemize}
\item[(a)] An integer $k>1$ will be called a {\it primitive element order} of $G$ if $G$ has a maximal cyclic subgroup of order $k$.
\item[(b)] An integer $k>1$ will be called a {\it maximal element order} of $G$ if $G$ has an element of order $k$, but has no element
whose order is a proper (finite) multiple of $k$.
\end{itemize}
We will denote the set of all primitive (resp. maximal) element orders of $G$ by $\PEO(G)$ (resp. $\MEO(G)$).
\end{Definition} 

\begin{Lemma}
\label{lem:orders1}
Let $G$ be a discrete group. Then $\MEO(G)$ is a subset of $\PEO(G)$. Moreover, $|\PEO(G)|\leq \CC(G,\Phi)$ for any $\Phi$.
\end{Lemma}
\begin{proof}
The first assertion is clear. If $\{C_i\}$ is any cyclic $\Phi$-cover of $G$,
then for any maximal cyclic subgroup $C$ of $G$, the $\Phi$-orbit of $C$ must contain one of the subgroups $C_i$,
so $\MC(G,\Phi)\leq \CC(G,\Phi)$ where $\MC(G,\Phi)$ is the number of $\Phi$-orbits of maximal cyclic subgroups.
On the other hand, if $C$ and $C'$ are maximal cyclic subgroups of different orders, they must be in different orbits.
Thus, $|\PEO(G)|\leq \MC(G,\Phi)$, which proves the second assertion.
\end{proof}

The next result can be used, in particular, to show that if $G$ is a group with finite NCC, then a normal subgroup of $G$ cannot decompose as a direct product of too many non-abelian simple groups. We thank the referee for simplifying our original proof.
 
\begin{Lemma} 
\label{lem:dp}
Let $H$ be a discrete or profinite group and $\Phi$ a group acting on $H$ by automorphisms.
Suppose that there exists an integer $e>1$
and elements $h_1,\ldots, h_k$ of $H$ with the following properties:
\begin{itemize}
\item[(i)] Each $h_i$ has order $e$.
\item[(ii)] For any $i\neq j$ there is no $\phi\in\Phi$ such that $\phi(\la h_i\ra)=\la h_j\ra$.
\end{itemize}
Then $\CC(H,\Phi)\geq k$.
\end{Lemma}
\begin{proof} Suppose that $\CC(H,\Phi)<k$. Then there exists $i\neq j$ and $\phi\in\Phi$ such that
$\phi(\la h_i\ra)$ and $\la h_j\ra$ lie in the same cyclic or procyclic subgroup $C$ of $H$. Since
$\phi(\la h_i\ra)$ and $\la h_j\ra$ both have order $e$ by (i) and
cyclic or procyclic groups have at most one subgroup of any given finite order, it follows that
$\phi(\la h_i\ra)=\la h_j\ra$, contrary to (ii). 
\end{proof}

\begin{Corollary} 
\label{cor:dp}
Let $H=S_1\times \cdots \times S_k$ where $S_i$ are non-abelian finite simple groups (not necessarily distinct). Then
$\CC(H,\Phi)\geq k$ for any group $\Phi$ acting on $H$ by automorphisms.
\end{Corollary} 
\begin{proof} We will only use the fact that each $S_i$ is a finite group of even order and has trivial center.

Choose elements $s_i\in S_i$ of order $2$, and for each $1\leq i\leq k$ let $h_i=(s_1,s_2,\ldots, s_i,1,\ldots,1)$.
Since each $s_i$ is non-central in $S_i$, the sequence of centralizers $C(h_1)\supset C(h_2)\supset \cdots \supset C(h_k)$
is strictly decreasing. Hence the elements $\{h_i\}$ lie in different $\Phi$-orbits. Since $\{h_i\}$ have prime order (namely order $2$),
they satisfy the hypotheses of Lemma~\ref{lem:dp} and hence $\CC(H,\Phi)\geq k$.
\end{proof}

Before stating our last result of this section, we introduce one more definition.

\begin{Definition}\rm
Let $G$ be a discrete or profinite group. We will say that $G$ has {\it property (FMHFG)}\footnote{(FMHFG) stands for `finitely many homomorphisms to a finite group'.} 
if for any finite group $F$ there are only finitely many homomorphisms from $G$ to $F$.
\end{Definition}
\begin{Remark}\rm A discrete (resp. profinite) group $G$ has (FMHFG) if and only if for any $i\in\dbN$ it has finitely many subgroups (resp. open subgroups) of index $i$.
\end{Remark}

Clearly finitely generated groups have (FMHFG). A simple example of an infinitely generated group with (FMHFG) is the direct sum or product
of an infinite collection of finite groups of pairwise coprime orders.  

\begin{Lemma}
\label{lem:FMHFG}
Let $G$ be a discrete or profinite group with finite NCC. Then $G$ has (FMHFG).
\end{Lemma}
\begin{Remark}\rm We will eventually prove that profinite and discrete residually finite groups with finite NCC are finitely generated. However, Lemma~\ref{lem:FMHFG} will be needed as an auxiliary tool in order to establish finite generation. 
\end{Remark}
\begin{proof}[Proof of Lemma~\ref{lem:FMHFG}] Fix a finite group $F$, and let $K$ be the intersection of the kernels of all homomorphisms from $G$ to $F$. Then any homomorphism
from $G$ to $F$ factors through $G/K$, so it suffices to prove that $G/K$ is finite.

Clearly $G/K$ embeds into a direct power $H=\prod\limits_{i\in I}F$ for some index set $I$. Since $H$ and hence $G/K$ is torsion, all cyclic subgroups of $H$ are closed, so there is no need to distinguish between the discrete and profinite cases. For any element $h\in H$ and $i\in I$
we denote by $h_i$ the $i^{\rm th}$ coordinate of $h$.

Take any $h\in H$, let $e=\ord(h)$, and let $I(h)$ be any finite subset of $I$ such that $LCM(\{\ord(h_i): i\in I(h)\})=e$. Then if some 
$x\in H$ lies in a conjugate of $\la h\ra$ and $x_i=1$ for all $i\in I(h)$, we must have $x=1$. 
Since $G/K$ has finite NCC, it lies in the union of conjugacy classes of finitely many cyclic subgroups $\la h_1\ra,\ldots, \la h_k\ra$.
If $J=\bigcup\limits_{i=1}^k I(h_k)$, then any $g\in G/K$ such that $g_j=1$ for all $j\in J$ must be trivial. But this means that $G/K$
embeds into the finite group $\prod\limits_{j\in J}F$, as desired.
\end{proof}

\subsection{NCC of a profinite group and its finite quotients}

In this last subsection we will explain why the existence of a family of non-cyclic finite $p$-groups with bounded NCC implies
the existence of a non-procyclic pro-$p$ group with finite NCC (see Claim~\ref{claim:finitefamily} below). But first we need
to establish the following standard lemma.

\begin{Lemma}
\label{NCC:profinite} Let $G=\varprojlim\limits_{i\in I} P_i$ where $\{P_i\}_{i\in I}$ is an inverse system of finite groups in which all the maps $P_i\to P_j$ are surjective. Then $$\NCC(G)=\sup \NCC(P_i).$$ In particular, for any profinite group $G$ we have
$\NCC(G)=\sup \NCC(P)$ where $P$ ranges over all finite quotients of $G$.
\end{Lemma}
\begin{proof} By \cite[Proposition~1.4]{DDMS}, the inverse limit of a system of compact (in particular, finite) sets $P_i$ is always non-empty. Moreover, the proof shows that if all the maps $P_i\to P_j$ are surjective, then so is the induced map $\varprojlim\limits_{i\in I} P_i\to P_j$. Thus, in our setting $\NCC(G)\geq \NCC(P_i)$ for each $i$, and so $\NCC(G)\geq \sup \NCC(P_i)$.

To prove the reverse inequality $\NCC(G)\leq \sup \NCC(P_i)$ we just need to show that if $k\in\dbN$ is such that $\NCC(P_i)\leq k$ for all
$i$, then $\NCC(G)\leq k$. Take any such $k$, and for each $i\in I$ let $S_{i}$ be the set of all sequences $(g_i(1),\ldots, g_i(k))\in P_i^k$
such that the conjugacy classes of the cyclic subgroups $\la g_i(1)\ra,\ldots, \la g_i(k)\ra$ cover $P_i$. By the choice of $k$ each $S_i$ is non-empty. Moreover, the sets $\{S_i\}$ form an inverse system with the maps $S_i\to S_j$ defined componentwise. 
By \cite[Proposition~1.4]{DDMS}, the inverse limit $S=\varprojlim\limits_{i\in I} S_i$ is non-empty; on the other hand, we can naturally
identify $S$ with a subset of $G^k$. Let $(g(1),\ldots, g(k))$ be any element of $S$, and let $T=\bigcup\limits_{i=1}^k 
\overline{\la g(i)\ra}^G$ (where $A^G$ denotes the normal closure of $G$). Then $T$ is a closed
subset of $G=\varprojlim\limits_{i\in I} P_i$ which projects onto each $P_i$, and from this it is easy to deduce that $T=G$.
Thus $\NCC(G)\leq k$, as desired.
\end{proof}

\begin{Claim}
\label{claim:finitefamily}
Suppose that for some $k$ there exists an infinite sequence of noncyclic finite $p$-groups $\{P_i\}$ with $\NCC(P_i)\leq k$ for all $i$. Then there exists an infinite non-procyclic pro-$p$ group $G$ with $\NCC(G)\leq k$. 
Moreover, if $d(P_i)=d$ for all $i$, we can assume that $d(G)=d$.
\end{Claim}
\begin{proof} 
First observe that if $P$ is a finite $p$-group and $d=d(G)$, then $P/\Phi(P)\cong (\dbZ/p\dbZ)^d$ 
whence $\NCC(P)\geq \NCC((\dbZ/p\dbZ)^d)=\frac{p^d-1}{p-1}$. Hence for any family of finite $p$-groups with bounded NCC, 
the sequence $\{d(P_i)\}$ is also bounded. Thus, it suffices to prove Claim~\ref{claim:finitefamily} assuming that 
there exists $d\in\dbN$ such that $d(P_i)=d$ for all $i$.

Consider the following oriented graph $\Gamma_{k,d}(p)$. The vertices of $\Gamma_{k,d}(p)$ are (isomorphism classes of) 
finite $p$-groups $P$ with $d(P)=d$ and $\NCC(P)\leq k$ (thus by our hypothesis $\Gamma_{k,d}(p)$ is infinite). There is an oriented edge from $P$
to $Q$ if and only if $Q\cong P/Z$ where $|Z|=p$ and $Z\subseteq \Phi(P)$.
 
Any finite $p$-group $P$ with $d(P)=d$ and $P\not\cong (\dbZ/p\dbZ)^d$ has a central subgroup $Z$ of order $p$ lying in $\Phi(P)$.
Therefore, for any such $P$ there is a directed path from $P$ to $(\dbZ/p\dbZ)^d$ in $\Gamma_{k,d}(p)$. In particular
$\Gamma_{k,d}(p)$ is connected and thus contains an infinite path $Q_1\leftarrow Q_2\leftarrow Q_3\leftarrow\cdots$.
Let $G=\varprojlim Q_i$. Since $d(Q_i)=d$ for all $i$, we have $d(G)=d$. Also by Lemma~\ref{NCC:profinite}, 
$\NCC(G)=\sup \{\NCC(Q_i)\}$, so $\NCC(G)\leq k$, as desired.  
\end{proof}

\section{Reduction to the residually solvable case}

Recall that by $\Nil$ and $\Sol$ we denote the classes of finite nilpotent and finite solvable groups, respectively.
The goal of this section is to establish the first of the three parts in the proof of Theorem~\ref{thm:mainprofinite} (recall that the three parts were introduced in \S~\ref{sec:outline}). 

\begin{Theorem} 
\label{reduction:ressolvable}
Let $G$ be a profinite (resp. a discrete residually finite) group, and assume that $\NCC(G)<\infty$. Then 
$G$ is virtually pro-$\Sol$ (resp. virtually residually-$\Sol$).
\end{Theorem}

Theorem~\ref{reduction:ressolvable} in the discrete case immediately follows from its profinite analogue. 
Indeed, let $G$ be a discrete residually finite group with finite NCC. Then its profinite completion $\widehat G$ is a profinite group with finite NCC.
By the profinite part of Theorem~\ref{reduction:ressolvable}, $\widehat G$ has an open pro-$\Sol$ subgroup $U$, and so $G\cap U$ is a finite index residually-$\Sol$ subgroup of $G$.

Thus, it suffices to prove Theorem~\ref{reduction:ressolvable} for a profinite group $G$. This will be done by analyzing the action of $G$ on the factors of its chief series defined as follows.

\begin{Definition}\rm
Let $G$ be a profinite group. A descending chain of open normal subgroups $G=G_1\supseteq G_2\supseteq\cdots$ will be called a {\it chief series}  of $G$ if the following hold:
\begin{itemize}
\item[(i)] $\{G_i\}$ is a base of neighborhoods for the topology on $G$. Since $G$ is profinite,
this is equivalent to requiring that $\cap G_i=\{1\}$.
\item[(ii)] $G$ does not have any normal subgroups lying strictly between $G_i$ and $G_{i+1}$. 
\end{itemize} 
\end{Definition}

Note that a profinite group $G$ has a series satisfying (i) if and only if it is countably based.
Moreover, if we start with any series $\{G_i\}$ satisfying (i), then (ii) can always be achieved
by refining the series (since each $G_i/G_{i+1}$ is finite and hence has a chief series in the usual sense).
Recall that by Lemma~\ref{lem:FMHFG} groups with finite NCC have property (FMHFG) which is equivalent to having
finitely many open subgroups of any given index $i$. Thus, groups with finite NCC are countably based and
hence admit a chief series.

\begin{Observation} 
\label{obs:chief}
Let $G$ be a countably based profinite group. The following hold:
\begin{itemize}
\item[(a)] $G$ is pro-$\Nil$ if and only if it admits a chief series
$\{G_i\}$ such that $G$ acts trivially on each quotient $G_i/G_{i+1}$. 
\item[(b)] $G$ is pro-$\Sol$ if and only if it admits a chief series
$\{G_i\}$ such that each quotient $G_i/G_{i+1}$ is abelian.
\end{itemize}
\end{Observation}
\begin{Remark}\rm It is also easy to show that $G$ is pro-$\Nil$ (resp. pro-$\Sol$) if and only if
every chief series of $G$ satisfies the extra condition in (a) (resp. (b)).
\end{Remark}

For the rest of this section we fix a profinite group $G$ with (FMHFG) and also fix a chief series $\{G_i\}$ of $G$. For each $i$ let $Q_i=G_i/G_{i+1}$. We know that $Q_i\cong S_i^{n_i}$ for some finite simple group $S_i$ and $n_i\in\dbN$.

\begin{Lemma}
\label{lem:bounded} Whenever $S_i$ is non-abelian we have $n_i\leq \NCC(G)$.
\end{Lemma}
\begin{proof}
By Corollary~\ref{cor:dp} we have $\CC(Q_i,G)\geq n_i$ (where $\CC$ is with respect to the conjugation action of $G$ on $Q_i$),
and by Lemma~\ref{CC:subgroup}(i)(ii) $\CC(Q_i,G)\leq \NCC(G)$. 
\end{proof}

\begin{Lemma} 
\label{lem:repetition}
For any non-abelian simple group $S$ there are only finitely many $i$ such that $S_i=S$. 
\end{Lemma}
\begin{proof} Fix $S$. Since $G$ has (FMHFG), there are only finitely many homomorphisms from $G$ to the finite group $\Aut(S^{\NCC(G)})$.
Let $H$ be the intersection of the kernels of these homomorphisms. Then $H$ is an open 
subgroup of $G$. By Lemma~\ref{lem:bounded}, for any $i$ with $S_i=S$, the group $\Aut(Q_i)$ embeds into
$\Aut(S^{\NCC(G)})$, whence $H$ acts trivially on $Q_i$ and thus cannot contain $G_i$ for any such $i$ (since $G_i$ acts non-trivially on $Q_i$
as $S_i$ is non-abelian). On the other hand, since $H$ is open, it must contain $G_j$ for some $j$, so we can only have $S_i=S$
for $i<j$. 
\end{proof}

We are now ready to prove Theorem~\ref{reduction:ressolvable}. In view of Observation~\ref{obs:chief}(b), the result can be reformulated
as follows.

\begin{Proposition} 
\label{reduction:ressolvable0}
Assume that $\NCC(G)<\infty$. Then $S_i$ is abelian for all sufficiently large $i$ and therefore $G$ is virtually pro-$\Sol$.
\end{Proposition}
\begin{proof}
Let $I$ be the set of all $i$ such that $S_i$ is non-abelian. Our goal is to show that $I$ is finite.
First we want to reduce the problem to the case where $n_i=1$ for all $i\in I$.

For each $i$ the conjugation action of $G$ on $Q_i=S_i^{n_i}$
induces a homomorphism $\pi_i:G\to Sym(n_i)$. Since $n_i\leq \NCC(G)$ by Lemma~\ref{lem:bounded} and $G$ has (FMHFG) by Lemma~\ref{lem:FMHFG}, 
there are only finitely many such homomorphisms.
If $H$ is the intersection of the kernels of these homomorphisms, then $H$ is an open subgroup of $G$ which preserves each direct factor
of each $Q_i$. Thus, $H$ has a chief series (obtained by a refinement of the series $\{H\cap G_i\}$) where all non-abelian chief factors are
simple. 

Thus, replacing $G$ by $H$ (which also has finite NCC) we can assume that $n_i=1$ for $i\in I$, as desired. Under this extra assumption, for each $i\in I$ we have a homomorphism
$\pi_i:G\to \Aut(S_i)$. 

For any finite simple group $S$, the outer automorphism group $\Out(S)$ is solvable of derived length at most $3$.
This follows from the classification of finite simple groups and the explicit description of $\Out(S)$ for every finite simple group $S$ -- see \cite[Theorem~2.5.12~and~7.1.1(a)]{GLS}. 
Thus, if  $K=G^{(3)}$ is the third (closed) derived subgroup of $G$,
then $\pi_i(K)\subseteq \Inn(S_i)$ for all $i\in I$. In fact, we have $\pi_i(K)= \Inn(S_i)$ for all $i\in I$. Indeed, $\pi_i(G)$ contains $\pi_i(S_i)=\Inn(S_i)$ and hence
$\pi_i(K)$ contains $\Inn(S_i^{(3)})=\Inn(S_i)$ (since $S_i$ is perfect).
\skv

Identifying $\Inn(S_i)$ with $S_i$, we can reformulate the conclusion of the previous paragraph as follows. For every $i\in I$
there exists an epimorphism $\phi_i:K\to S_i$ which is $G$-equivariant with respect to the action of $G$ on $S_i$ given by $\pi_i$
and the conjugation action of $G$ on $K$. 

Now take any finite subset $J\subseteq I$ such that the groups $\{S_j\}_{j\in J}$ are pairwise non-isomorphic
and consider the diagonal map
$\phi:K\to\prod\limits_{j\in J}S_j$, which is also $G$-equivariant. By construction $\phi(K)$ surjects onto each direct factor $S_j$,
and since $\{S_j\}$ are pairwise non-isomorphic non-abelian finite simple groups, $\phi$ is surjective. 
Hence $\NCC(G)\geq \CC(K,G)\geq \CC(\prod\limits_{j\in J}S_j,G)\geq |J|$ where the first inequality holds by 
Lemma~\ref{CC:subgroup}(i), the second is immediate from the $G$-equivariance of $\phi$ and the third one holds by
Corollary~\ref{cor:dp}. Since $\NCC(G)<\infty$, we proved that the collection $\{S_i\}_{i\in I}$ contains only finitely
many pairwise non-isomorphic groups. Combined with Lemma~\ref{lem:repetition}, this implies that
$I$ is finite, as desired.
\end{proof}

\section{Reduction to the residually nilpotent case}

{\bf Notation:} Given a discrete group $G$ we will denote by $\{G^{(i)}\}_{i=0}^{\infty}$ its derived series, that is, define the subgroups $G^{(i)}$ inductively by $G^{(0)}=G$ and $G^{(i)}=[G^{(i-1)},G^{(i-1)}]$ for $i\geq 1$.
If $G$ is profinite, $\{G^{(i)}\}$ will denote the closed derived series, that is,
$G^{(i)}=\overline{[G^{(i-1)},G^{(i-1)}]}$ for $i\geq 1$.

In this section we will complete the second part of the proof of Theorem~\ref{thm:mainprofinite} by establishing the following result. 

\begin{Theorem}
\label{reduction:resnilpotent}
Let $G$ be a pro-$\Sol$ (resp. a discrete residually-$\Sol$) group, and assume that $\NCC(G)<\infty$. Then 
there exists $k\in\dbN$ such that $G^{(k)}$ is pro-$\Nil$ (resp. residually-$\Nil$).
\end{Theorem}

Similarly to Theorem~\ref{reduction:ressolvable}, it suffices to prove
Theorem~\ref{reduction:resnilpotent} for pro-$\Sol$ groups.

The third and final
part of the proof of Theorem~\ref{thm:mainprofinite} is fairly long and will be postponed till \S~10. However, Theorems~\ref{reduction:ressolvable}~and~\ref{reduction:resnilpotent} and Lemma~\ref{lem:Z} are sufficient to deduce the counterpart of Theorem~\ref{thm:mainprofinite} for finitely generated discrete residually finite groups:

\begin{Corollary}
\label{reduction:discretefg}
Let $G$ be a finitely generated discrete residually finite group with finite NCC. Then $G$ is virtually residually-$\Nil$. 
\end{Corollary}
\begin{proof}
By Theorems~\ref{reduction:ressolvable}~and~\ref{reduction:resnilpotent},
$G$ has a finite index subgroup $U$ such that $U^{(k)}$ is residually-$\Nil$ for some $k$. If $G$ is virtually cyclic, there is nothing to prove. If $G$ is not virtually cyclic, applying Lemma~\ref{lem:Z} $k$ times we deduce
that $U^{(k)}$ has finite index in $G$, which finishes the proof.
\end{proof}

We now begin the proof of Theorem~\ref{reduction:resnilpotent}. For the rest of the section we fix a pro-$\Sol$ group $G$ with finite NCC. By Observation~\ref{obs:chief}(b), $G$ admits a chief series $\{G_i\}$ such that all the quotients $Q_i=G_i/G_{i+1}$ are abelian. We will also fix such a chief series. For each $i$ we have $Q_i\cong \dbF_{p_i}^{n_i}$ for some prime $p_i$ and $n_i\in\dbN$.

We start by reducing Theorem~\ref{reduction:resnilpotent} to a certain result on solvable subgroups of linear groups over finite fields
(see Proposition~\ref{prop:solvlinear} below).

For each $i$ we can think of $Q_i$ as a finite-dimensional vector space over $\dbF_{p_i}$. To emphasize this point of view we will write $\GL(Q_i)$
instead of $\Aut(Q_i)$. Let $T_i$ denote the image of $G$ in $\GL(Q_i)$. Note that each $T_i$ must be solvable.
To prove Theorem~\ref{reduction:resnilpotent} it suffices to show that the derived lengths of the groups $T_i$ are bounded
by some $k\in\dbN$ (in fact, we will explicitly bound $k$ in terms of $\NCC(G)$). Indeed, if this is true, then
$G^{(k)}$ acts trivially on all chief factors $Q_i=G_i/G_{i+1}$ and hence also on their subgroups $(G_i\cap G^{(k)})/(G_{i+1}\cap G^{(k)})$
as well as on the chief factors of any chief series of $G^{(k)}$ refining $\{G_i\cap G^{(k)}\}_{i=1}^{\infty}$.
Hence $G^{(k)}$ must be pro-$\Nil$ by Observation~\ref{obs:chief}(a).

For each $i$ we have $\NCC(T_i)\leq \NCC(G)$. On the other hand, if $C$ is the conjugacy class of a cyclic subgroup of $G/G_{i+1}$, then the intersection
of $C$ with $Q_i$ is either trivial or is the orbit of a $1$-dimensional subspace of $Q_i$ under the action of $T_i$. Thus the action of $T_i$
on the set of $1$-dimensional subspaces of $Q_i$ has at most $\NCC(G)$ orbits. 

Let $T_i'$ be the subgroup of $\GL(Q_i)$ generated by $T_i$ and the scalar matrices. Then $T_i'$ is also solvable with $\ell(T_i)\leq \ell(T_i')$
where $\ell(\cdot)$ denotes the derived length (in fact, $\ell(T_i)= \ell(T_i')$ unless $T_i$ is the trivial group),
and the action of $T_i'$ on the set of nonzero elements of $Q_i$ has the same number of orbits as the action of $T_i$ on $1$-dimensional subspaces.
Thus, if we bound the derived length of $T_i'$ in terms of the number of orbits of its action on $Q_i\setminus\{0\}$, we will be done.
More precisely, we are now reduced to proving the following result: 

\begin{Proposition}
\label{prop:solvlinear}
Let $H$ be a solvable subgroup of $\GL_n(\dbF_p)$ for some prime $p$. Consider $H$ as an abstract group acting on $\dbF_p^n$,
and let $r$ be the number of orbits of this action. Then the derived length of $H$ is bounded above by $f(r)$ for some absolute function $f$ (independent of $p$ and $n$).  
\end{Proposition}

We need some preparation before proving Proposition~\ref{prop:solvlinear}.

\begin{Definition}\rm 
Let $P$ be a permutation group acting on a set $X$. 
\begin{itemize}
\item[(i)] Define $r(P)$ to be the number of orbits of $P$ on $X$.
\item[(ii)] The {\it rank of $P$}, denoted $\rk(P)$, is the number of orbits of the diagonal action of $P$ on $X\times X$.
\item[(iii)] The {\it degree of $P$} is the cardinality of $X$.
\end{itemize}
\end{Definition}

The following result is well known, but for completeness we provide a sketch of proof.

\begin{Lemma} 
\label{lem:permutation}
Let $H$ be a subgroup of $\GL(V)$ for some nonzero vector space $V$ over a field $F$, and let $AH$ be the group generated by $H$
and all translations $x\mapsto x+v$ with $v\in V$ (so $AH$ is a subgroup of the affine group $\AGL(V)$).
The following hold:
\begin{itemize}
\item[(a)] $r(H)=\rk(AH)$.
\item[(b)] Assume that either $|F|$ is prime or $H$ contains all (nonzero) scalar operators.
Then $H$ is irreducible as a linear group (that is, $V$ has no non-trivial $H$-invariant subspaces) if and only if
$AH$ is primitive as a permutation group.
\end{itemize}
\end{Lemma}
\begin{proof}[Sketch of proof] (a) holds since $AH$ acts transitively on $V$ and $H$ is a point stabilizer in $AH$ (namely the
stabilizer of $0$).
\skv
(b) If $V$ contains a non-trivial $H$-invariant subspace $W$ then cosets of $W$ form a non-trivial $AH$-invariant partition
of $V$, so $AH$ is not primitive.

Suppose now that $AH$ is not primitive, and let $\Omega$ be a non-trivial $AH$-invariant partition of $V$. Let $W$
be the block of $\Omega$ containing $0$. Then $W$ is $H$-invariant since $H$ fixes $0$.
Since $AH$ contains all maps of the form $x\mapsto x+a$, $a\in V$, it is easy to show that $W$ is a subgroup of $V$ (and hence also a subspace if 
$|F|$ is prime). If $|F|$ is not prime, by assumption $AH$ contains all maps of the form 
$x\mapsto \lam x+a$, $a\in V$, $\lam\in F$ which similarly implies that $W$ is a subspace. Thus $H$ is not irreducible.
\end{proof} 

\begin{proof}[Proof of Proposition~\ref{prop:solvlinear}]
We first consider the case where $H$ is an irreducible subgroup of $\GL_n(\dbF_p)$. In this case Proposition~\ref{prop:solvlinear} easily follows from a theorem of Seager~\cite[Theorem~1]{Se}  whose simplified version is stated below:

\begin{Theorem}[\cite{Se}]
\label{thm:Seager}
Let $P$ be a solvable primitive permutation group of rank $r$ and degree $d$. Then one of the following holds:
\begin{itemize} 
\item[(i)] $d\leq f_1(r)$ for some absolute function $f_1$.
\item[(ii)] There exist a prime $p$ and integers $m$ and $k$ with $k\leq f_2(r)$ for some absolute function $f_2$
such that $P$ embeds into the permutation wreath product $S(p^m)\,{\rm  wr}_{[k]}\, S_k$. Here $[k]=\{1,\ldots, k\}$,
$S_k$ is the symmetric group on $[k]$ and $S(p^m)$ is the group of all maps $\dbF_{p^m}\to \dbF_{p^m}$ of the form $x\mapsto a\sigma(x)+b$
with $a,b\in \dbF_{p^m}$, $a\neq 0$ and $\sigma\in \Aut(\dbF_{p^m})$.
\end{itemize}
\end{Theorem}

Since $H$ is solvable, the group $AH$ (defined as in Lemma~\ref{lem:permutation}) is also solvable. Since $H$ is irreducible,
$AH$ is primitive, so we can apply Theorem~\ref{thm:Seager} to $P=AH$. If (i) holds, then the order of $P$ (and hence also its derived length) is bounded by a function of $r$, so we are done. Suppose now that (ii) holds. If $Q$ is the projection of $P$
to $S_k$, then $P$ embeds into $S(p^m)\,{\rm  wr}_{[k]}\, Q$, and since $P$ is solvable, so is $Q$.  
It is straightforward to check that $S(p^m)$ is solvable of derived length $\leq 3$, whence the derived length of the wreath product 
$S(p^m)\,{\rm wr}_{[k]}\, Q$ (and hence also the derived length of $P$) is bounded by a function of $k$ and
hence also by a function of $r$, as desired. Thus we proved Proposition~\ref{prop:solvlinear} when $H$ is irreducible.
\skv

Now consider the general case. Let $V=\dbF_p^n$, and let $\{0\}=V_0\subset V_1\subset\cdots \subset V_t=V$ be a maximal chain of $H$-invariant subspaces. Note that $t < r=r(H)$ since vectors lying in $V_i\setminus V_{i-1}$ and $V_j\setminus V_{j-1}$ for $i\neq j$ cannot lie
in the same orbit of $H$.

Let $H_i$ be the canonical image of $H$ in $\GL(V_i/V_{i-1})$. Then each $H_i$ is an irreducible solvable linear group with $r(H_i)\leq r$ and hence by Proposition~\ref{prop:solvlinear} in the irreducible case, its derived length $\ell(H_i)$ is bounded above by $f_{\rm{irr}}(r)$ for some 
absolute function $f_{\rm{irr}}$. On the other hand, the kernel $K$ of the natural projection $H\to\prod\limits_{i=1}^t H_i$
is a nilpotent group of class $\leq t-1$ (see the proof below). Hence $\ell(K)\leq \log_2(t-1)$, and therefore $\ell(H)\leq \ell(K)+\ell(\prod H_i)\leq \log_2(t-1)+\max \ell(H_i) <\log_2(r)+f_{\rm{irr}}(r)$, as desired.  
\skv

To prove that $K$ is nilpotent of class $\leq t-1$ notice that $K\subseteq 1+I$ where $I$ is the set of all $f\in {\rm End}(V)$ such that $f(V_j)\subseteq V_{j-1}$ for all $1\leq j\leq t$. Clearly $I$ is a ring (without $1$) and $I^t=0$, whence $1+I$ is a group. By direct computation $[1+I^j,1+I]\subseteq 1+I^{j+1}$ for all $j$. Hence $\gamma_t K\subseteq \gamma_t (1+I)=\{1\}$,
as desired.
\end{proof}

\section{Proof of Theorem~\ref{main:discrete}}

In this section we will prove Theorem~\ref{main:discrete}, assuming the other main results of this paper that will be proved later. 
We will also use the following immediate fact:

\begin{Observation}
\label{obs:basic} Let $G$ be a profinite group containing a dense subgroup $\Gamma$ which has finite NCC (as a discrete group).
Then $G$ has finite NCC. Hence the profinite and pro-$p$ completions of a discrete group with finite NCC
have finite NCC. 
\end{Observation}
\begin{proof} The first claim follows directly from definitions. The second claim follows from the first one and the fact that
finiteness of NCC is inherited by quotients.
\end{proof}

For the rest of the section we fix an infinite residually finite discrete group $G$ with finite NCC. Our goal is to show that $G$ is cyclic or dihedral. Thanks
to the following result, it will be sufficient to show that $G$ is virtually solvable:

\begin{Proposition}
\label{prop:virtsolv}
Let $G$ be an infinite virtually solvable discrete group with finite NCC. Then $G$ is cyclic or dihedral.
\end{Proposition}

Proposition~\ref{prop:virtsolv} is a direct combination of the main result of \cite{GW} which implies that
discrete virtually solvable groups with finite NCC are virtually cyclic and \cite[Proposition~3.8]{vPW2} which asserts that an infinite 
virtually cyclic group with finite NCC is cyclic or dihedral.

The rest of the proof will be divided into three steps, with the first step proving Theorem~\ref{main:discrete} in a special case
and each of the subsequent steps reducing to the previous one. In Steps~1~and~3, we will give a separate argument in the finitely
generated case using more elementary ingredients.
\skv
{\it Step 1: $G$ is residually-$p$ for some prime $p$}. In this case $G$ embeds in its pro-$p$ completion
$\widehat G_{p}$. The group $\widehat G_{p}$ is a pro-$p$ group with finite NCC and therefore $p$-adic analytic
by Theorem~\ref{thm:padic}. In particular, $\widehat G_{p}$ (and hence also $G$) is linear over $\dbQ_p$.

If $G$ is finitely generated, we can deduce that $G$ is cyclic or dihedral directly from the following theorem of 
Puttkamer and Wu~\cite{vPW2}:

\begin{Theorem}
\label{thm:NCClinear}
Let $H$ be an infinite finitely generated discrete linear\footnote{As usual, by a linear group we will mean a group embeddable in $\GL_n(F)$ for some field $F$ and
$n\in\dbN$.} group with finite NCC. Then $H$ is cyclic or dihedral.
\end{Theorem}

If $G$ is not necessarily finitely generated, we argue as follows. Let $\Lambda_p$ be the set of eigenvalues of elements of 
$G$ (with respect to a fixed embedding of $\widehat G_{p}$ into $\GL_n(\dbQ_p)$ for some $n\in\dbN$). Since $G$ has finite NCC, $\Lambda_p$ is a union of finitely many cyclic subgroups of $\overline{\dbQ_p}^{\times}$
(where $\overline{\dbQ_p}$ is the algebraic closure of $\dbQ_p$). In particular, $\Lambda_p$ lies in a finitely generated
subfield of $\overline{\dbQ_p}$. We can now deduce that $G$ is virtually solvable (thereby completing the proof 
of Theorem~\ref{main:discrete} for residually-$p$ groups) using the following theorem of Bernik~\cite{Be}:

\begin{Theorem}
\label{Bernik}
Let $A$ be a linear semigroup in characteristic zero such that the eigenvalues of all elements of $A$ lie in some finitely generated subfield. Then the subgroup generated by $A$ is virtually solvable.
\end{Theorem}
\skv

{\it Step 2: $G$ is residually-$\Nil$}. Then $G$ embeds in its pro-$\Nil$ completion $\widehat G_{\rm{nilp}}$ which is a pro-$\Nil$ group and thus is a direct product of its Sylow pro-$p$ subgroups $\widehat G_{p}$. Note that each $\widehat G_{p}$ 
is the pro-$p$ completion of $G$.

Let $G_p$ denote the image of $G$ in $\widehat G_{p}$. Then $G_p$ has finite NCC (being a quotient of $G$) and is residually-$p$,
so by Step~1 $G_p$ is finite, infinite cyclic or infinite dihedral; moreover, the last case may only occur when $p=2$ 
(since a residually-$p$ group cannot have $q$-torsion for $q\neq p$).
If $G_p$ is finite, it must be a finite $p$-group. If in addition $G_p$ is non-cyclic, its abelianization is also non-cyclic and hence
$\NCC(G_p)\geq p+1$. Since $\NCC(G_p)\leq \NCC(G)$, there are only finitely many $p$ for which $G_p$ is finite non-abelian. It follows that $G_p$ is abelian for almost all $p$ and virtually abelian for all $p$. Since $G$ embeds into $\prod G_p$, it must be virtually abelian,
and we are done.
\skv

{\it Step 3: $G$ is an arbitrary residually finite group}. Since finiteness of NCC is preserved by passing to finite index subgroups, by Step~2 it suffices to show that $G$ has a  finite index residually-$\Nil$ subgroup $H$. If $G$ is finitely generated, this holds by Corollary~\ref{reduction:discretefg}. In the general case, consider the profinite completion $\widehat G$. It has finite NCC and hence by  Theorem~\ref{thm:mainprofinite} has an open pro-$\Nil$ subgroup $U$. Then $H=G\cap U$ is a finite index residually-$\Nil$ subgroup of $G$.

\section{Pro-$p$ groups with finite NCC}

In this section we will prove Theorem~\ref{main:prop}.
Our first goal is to show that pro-$p$ groups with finite NCC are $p$-adic analytic (see Theorem~\ref{thm:padic} below).
We start with the definition of $p$-adic analytic groups and stating several characterizations of compact $p$-adic analytic groups.

\begin{Definition}\rm A topological group $G$ is called {\it $p$-adic analytic} if it can be given the structure
of a manifold over $\dbQ_p$ (compatible with the topology on $G$) such that the multiplication map $(x,y)\mapsto xy$
and the inversion map $x\mapsto x^{-1}$ are analytic.
\end{Definition}

It is quite remarkable that for compact groups, the property of being $p$-adic analytic is equivalent to other natural conditions of very different flavor, some of which are collected in the following theorem. We refer
the reader to \cite{DDMS} for the proof of this theorem and other characterizations of $p$-adic analytic groups.

\begin{Theorem}
\label{thm:padicbasis} Let $G$ be a compact topological group. The following are equivalent:
\begin{itemize}
\item[(a)] $G$ is $p$-adic analytic;
\item[(b)] $G$ is isomorphic to a closed subgroup of $\GL_n(\dbZ_p)$ for some $n\in\dbN$;
\item[(c)] $G$ is virtually pro-$p$ and has finite rank, that is, there exists $d\in\dbN$ such that
$d(H)\leq d$ for every closed subgroup $H$ of $G$. 
\end{itemize}
\end{Theorem}

\begin{Remark}\rm The equivalence of (a) and (b) in Theorem~\ref{thm:padicbasis} immediately implies that closed subgroups of $p$-adic analytic
groups are $p$-adic analytic. It also implies that compact $p$-adic analytic groups are virtually torsion-free. Indeed, it suffices
to prove the latter for $\GL_n(\dbZ_p)$, and an easy direct computation shows that the $k^{\rm th}$ congruence subgroup
$\GL_n^k(\dbZ_p)=\Ker(\GL_n(\dbZ_p)\to \GL_n(\dbZ_p/p^k\dbZ_p)$ is torsion-free if either $p>2$ or $k\geq 2$.
\end{Remark}

In order to prove that pro-$p$ groups with finite NCC are $p$-adic analytic we will use another important characterization,
which deals with the dimension subgroups.

Given a group $G$, let $\{D_n\}_{n=1}^{\infty}$ be the dimension series of $G$ in characteristic $p$. It is defined by $D_n=D_n(G)=\{g\in G: g\equiv 1\mod I^n\}$,
where $I$ is the augmentation ideal of the group algebra $\dbF_p[G]$, and has the following properties:
\begin{itemize}
\item[(a)] $[D_n,D_m]\subseteq D_{n+m}$ for all $n,m\in\dbN$.
\item[(b)] $D_n^p\subseteq D_{np}$ for all $n\in\dbN$.
\item[(c)] $G$ is residually-$p$ if and only if $\bigcap\limits_{n\in N}D_n=\{1\}$.
\end{itemize}
In fact, $\{D_n\}$ is the fastest descending chain of subgroups satisfying (a) and (b), but this will not be important
for our purposes. If $G$ is a finitely generated pro-$p$ group, it is not difficult to show that each $D_n$ is open in $G$ 
(see, e.g. \cite[\S~11]{DDMS}). 

We will use the well-known characterization of $p$-adic analytic pro-$p$ groups in terms of their dimension series (see, e.g. \cite[\S~11]{DDMS}):

\begin{Theorem}
\label{thm:dimension}
Let $G$ be a finitely generated pro-$p$ group $G$. Then $G$ is $p$-adic analytic if and only if $D_n(G)=D_{n+1}(G)$ for some $n\in\dbN$.
\end{Theorem}

We are now ready to prove Theorem~\ref{thm:padic}:

\begin{Theorem}
\label{thm:padic}
Any pro-$p$ group with finite NCC is $p$-adic analytic.
\end{Theorem}

\begin{proof} Fix a pro-$p$ group $G$. For any $1\neq x \in G$ define $\deg(x)$ to be the unique integer $n$ such that
$x\in D_n\setminus D_{n+1}$ (such $n$ exists by (c) above). Also set $\deg(1)=\infty$.
The following 3 properties of degree are straightforward:
\begin{itemize}
\item[(i)] Conjugate elements have the same degree (this holds by (a) above with $m=1$).
\item[(ii)] $\deg(x^p)\geq p\deg(x)$ for all $x\in G$ (this holds by (b)).
\item[(iii)] If $\lam\in\dbZ_p^{\times}$ is a unit of $\dbZ_p$, then $\deg(x)=\deg(x^{\lam})$ for all $x\in G$
(since in this case $x$ and $x^{\lam}$ generate the same procyclic subgroup).
\end{itemize}

Let us now assume that $G$ has finite NCC, so there exists a finite subset $\{x_1,\ldots,x_k\}$ of $G$ such that every
element of $G$ is conjugate to $x_i^{\lam}$ for some $1\leq i\leq k$ and $\lam\in\dbZ_p$. Let $d_i=\deg(x_i)$ (without loss of generality we can assume that $x_i\neq 1$, so $d_i<\infty$), and more generally let $d_{i,j}=\deg(x_i^{p^j})$.

Property (iii) above implies that for each $\lam\in\dbZ_p$ we have $\deg(x_i^{\lam})=d_{i,j}$ for some $j$ and hence
by (i) (and the choice of $\{x_1,\ldots,x_k\}$), the degree of any nonzero element of $G$ is equal to $d_{i,j}$
for some $i$ and $j$. 

On the other hand, $d_{i,j}\geq p^j d_i$ by (ii), so for each $N\in\dbN$ there are at most 
$k(\lfloor \log_p(N)\rfloor +1)$ possible degrees of elements of $G$ which are $\leq N$. Since 
$k(\lfloor \log_p(N)\rfloor +1)< N$ for large enough $N$, there exists $n\in\dbN$ which is not the degree of any element
of $G$. But this means precisely that $D_n(G)=D_{n+1}(G)$ and hence $G$ is $p$-adic analytic by Theorem~\ref{thm:dimension}.
\end{proof}

Our next result shows that a compact $p$-adic analytic group with finite NCC must have an element with small centralizer.

\begin{Proposition}
\label{prop:onedim}
Let $G$ be a compact $p$-adic analytic group with finite NCC. Then there exists $g\in G$ whose centralizer $C(g)$ is one-dimensional.
\end{Proposition}

In order to prove Proposition~\ref{prop:onedim}, we need a simple lemma:

\begin{Lemma} 
\label{lem:padicdim}
Let $X$ and $Y$ be $p$-adic manifolds and $\psi:X\to Y$ an analytic map whose
image has non-empty interior. Then $\dim X\geq \dim Y$.
\end{Lemma}
\begin{proof} It is not hard to prove Lemma~\ref{lem:padicdim} directly, but it also follows immediately from
Sard's Lemma for $p$-adic manifolds, as we now explain.

Let $K$ be the set of critical points of $\psi$, that is, the set of all $x\in X$
such that the derivative map $D_x(\psi):T_x(X)\to T_{\psi(x)}Y$ is not surjective (where
$T_x(X)$ and $T_{\psi(x)}Y$ are the tangent spaces). By Sard's Lemma over $\dbQ_p$ \cite[Theorem~2.3.3]{BKL},
$\psi(K)$ has measure zero in $Y$ and thus cannot have non-empty interior. 
Hence $K\neq X$, and for any $x\in X\setminus K$ we have $\dim X=\dim T_x(X)\geq T_{\psi(x)}Y=\dim Y$, as desired.
 \end{proof}

\begin{proof}[Proof of Proposition~\ref{prop:onedim}] By assumption there exist finitely many elements $x_1,\ldots, x_k\in G$ such that
$G=\cup_{i=1}^k \overline{\la x_i\ra}^G$, and assume that $k$ is smallest possible.
As in the proof of Lemma~\ref{lem:NCCrest}, each set $\overline{\la x_i\ra}^G$ is closed and
hence $G\setminus\cup_{j\neq i}\overline{\la x_j\ra}^G$ is open. By the minimality of $k$
the latter set is also non-empty, so $\overline{\la x_i\ra}^G$ has non-empty interior.

We can think of the conjugation map $\phi:(y,g)\mapsto g^{-1}yg$ restricted to $\overline{\la x_i\ra} \times G$ as a map
$\phi_i:\overline{\la x_i\ra} \times G/C(x_i)\to G$. Since $G$ is $p$-adic analytic and $C(x_i)$
is a closed subgroup, the quotient $G/C(x_i)$ is a $p$-adic manifold, and it is straightforward to
check that $\phi_i$ is an analytic map. 

Since $\Im(\phi_i)=\overline{\la x_i\ra}^G$, Lemma~\ref{lem:padicdim} is applicable, so $\dim (\overline{\la x_i\ra} \times G/C(x_i))\geq \dim(G)$. Since
$\dim (\overline{\la x_i\ra} \times G/C(x_i))=\dim (\overline{\la x_i\ra}) +\dim(G)-\dim C(x_i)\leq 1+\dim(G)-\dim C(x_i)$, 
we deduce that $\dim C(x_i)\leq 1$ for each $i$. It remains to show that there exists $i$ such that
$\dim C(x_i)\geq 1$, and the latter is definitely true if $x_i$ is not torsion (since  $C(x_i)$ contains $\overline{\la x_i\ra}$).

Finally, at least one $x_i$ is not torsion since otherwise $G$ is torsion which is impossible since $G$ is infinite and virtually torsion-free.
\end{proof}

To each $p$-adic analytic group $G$ one can associate a $\dbQ_p$-Lie algebra $L(G)$ with $\dim L(G)=\dim(G)$
such that $L(G)$ depends only on the commensurability class of $G$. For a classical definition of $L(G)$
we refer the reader to Serre's book~\cite{Ser2}, but for us it will be more convenient to follow the approach in \cite{DDMS}
(which, in turn, is based on Lazard's manuscript~\cite{La}) and
define $L(G)$ in terms of a certain $\dbZ_p$-Lie subalgebra which can be associated to any {\it uniform} pro-$p$ group.

A pro-$p$ group $G$ is called {\it powerful} if $[G,G]\subseteq \overline{G^{\mathbf p}}$ where 
$\mathbf p$ equals $p$ if $p>2$ and $4$ if $p=2$ and $G^{\mathbf p}$
is the subgroup generated by $\mathbf p^{\rm th}$ powers. A pro-$p$ group
$G$ is {\it uniform} if it is both powerful and torsion-free. Powerful (in particular, uniform) pro-$p$ groups are always $p$-adic analytic. Conversely, every $p$-adic analyic group contains an open uniform subgroup \cite[Corollary~8.34]{DDMS}. 
To any uniform pro-$p$ group $G$ one can associate a $\dbZ_p$-Lie algebra $L_G$ 
(see \cite[\S~6,7]{DDMS} for a proof):

\begin{Proposition} 
\label{prop:logLiealgebra}
Let $G$ be a uniform pro-$p$ group. There exists a structure of a normed $\dbQ_p$-algebra on $\dbQ_p[G]$ with the following properties:
\begin{itemize}
\item[(a)] Let $\widehat{\dbQ_p[G]}$ be the completion of $\dbQ_p[G]$ (with respect to the chosen norm). Then
the function $\log:G\to \widehat{\dbQ_p[G]}$ given by $$\log(g)=\sum_{i=1}^{\infty} \frac{(-1)^{i-1}}{i}(g-1)^i$$
is well-defined (that is, the series converges) and injective. Moreover, $\log$ is a bi-analytic homeomorphism onto its image.
\item[(b)] $L_G=\log(G)$ is a $\dbZ_p$-Lie subalgebra of $\widehat{\dbQ_p[G]}$ (with respect to the commutator bracket).
\end{itemize} 
\end{Proposition}

Given a $p$-adic analytic group $G$, one can now define $L(G)$ as follows: choose any open uniform pro-$p$ subgroup $H$
and set $L(G)=\dbQ_p\otimes_{\dbZ_p}L_H$. In \cite[\S~9]{DDMS} it is shown that $L(G)$ defined in this way is independent
of the choice of $H$ (up to isomorphism) and moreover is isomorphic to the Lie algebra of $G$ as defined in \cite{Ser2}.
\skv
We will need the following basic properties of $L_G$ and the $\log$ map defined above.

\begin{Proposition} 
\label{prop:padicLA}
Let $G$ be a uniform pro-$p$ group and $\exp:L_G\to G$ the inverse of the map $\log:G\to L_G$. The following hold:

\begin{itemize}
\item[(i)] Let $I$ be an ideal of $L_G$ such that $L_G/I$ is torsion-free. Then $\exp(I)$ is a (closed) normal subgroup of $G$.
\item[(ii)] Let $x,y\in G$. Then $x$ and $y$ commute if and only if $[\log(x),\log(y)]=0$. 
 \end{itemize}
\end{Proposition}
\begin{proof} (i) holds by \cite[Proposition~7.15]{DDMS}.

(ii) The Lie algebra $L_G$ admits an alternative definition (see \cite[\S~4]{DDMS} for that definition and \cite[Corollary~7.14]{DDMS} for the proof of its equivalence to the definition of $L_G$ given above). According to this alternative definition and \cite[Lemma~7.12]{DDMS} we have $[\log(x),\log(y)]=\lim\limits_{n\to\infty}\frac{1}{2n}\log([x^{p^n},y^{p^n}])$ for all $x,y\in G$ which immediately implies the forward direction.

To prove the backwards direction, take any $u,v\in L_G$ with $[u,v]=0$. We need to show that $\exp(u)$ and $\exp(v)$ commute.
In \cite[\S~6]{DDMS}, it is proved that $\exp(u)\cdot \exp(v)=\exp(\Phi(u,v))$ where $\Phi(u,v)=\sum\limits_{i=1}^{\infty}f_i(u,v)$,
each $f_i(u,v)$ is a homogeneous Lie polynomial in $u$ and $v$ of degree $i$ and $f_1(u,v)=u+v$.
Since $[u,v]=0$, it follows that $f_i(u,v)=0$ for $i>1$ whence $\exp(u)\cdot \exp(v)=\exp(u+v)=\exp(v+u)=\exp(v)\cdot \exp(u)$,
as desired.
\end{proof}

As an easy consequence of Proposition~\ref{prop:onedim}, Corollary~\ref{lem:ji} and Proposition~\ref{prop:padicLA}, we deduce
that for a pro-$p$ group $G$ with finite NCC, there are very few possibilities for $L(G)$:

\begin{Corollary} 
\label{cor:NCCLA}
Let $G$ be an infinite pro-$p$ group with finite NCC (so that $G$ is $p$-adic analytic
by Theorem~\ref{thm:padic}). Then $L(G)$ is isomorphic to $\dbQ_p$, $\sl_2(\dbQ_p)$ or $\sl_1(D)$
where $D$ is the quaternion division algebra over $\dbQ_p$. 
\end{Corollary}
\begin{proof} After passing to an open subgroup, we can assume that $G$ is uniform.
We claim that $L(G)$ has no nonzero proper $\dbQ_p$-ideals (that is, ideals which are also $\dbQ_p$-subspaces). 
Indeed, suppose that $I$ is a nonzero $\dbQ_p$-ideal of $L(G)$.
Then $L_G\cap I$ is a nonzero ideal of $L_G$ and $L_G/(L_G\cap I)$ is torsion-free (as it embeds in $L(G)/I$),
so by Proposition~\ref{prop:padicLA}, $N=\exp(L_G\cap I)$ is a non-trivial normal subgroup of $G$.
Since $G$ is just-infinite by Corollary~\ref{lem:ji}, $N$ is open in $G$ whence $L_G\cap I$ is open in $L_G$.
Since $L_G/(L_G\cap I)$ is also torsion-free, it must be trivial. Thus $I$ contains $L_G$ and hence $I=L(G)$.

Thus we proved that $L(G)$ is either one-dimensional (and thus isomorphic to $\dbQ_p$) or simple (non-abelian). Let us proceed
with the latter case.
\skv
By Proposition~\ref{prop:onedim}, there exists $g\in G$ with $\dim C_G(g)=1$.
By Proposition~\ref{prop:padicLA}(ii) we have $\log(C_G(g))=C_{L_G}(\log(g))$. Since $\log:G\to L_G$
is bi-analytic, Lemma~\ref{lem:padicdim} implies that $\dim C_{L_G}(\log(g))=\dim C_G(g)=1$. Since
$C_{L(G)}(\log(g))=\dbQ_p C_{L_G}(\log(g))$, it follows that $\dim C_{L(G)}(\log(g))=1$ as well.

Let $r$ denote the rank of $L(G)$. By one of the definitions of the rank, 
$r$ is the minimal value of $\dim C_{L(G)}(x)$ as $x$ ranges over $L(G)$, so we must have $r=1$. Finally, it is well known that  there are only two simple Lie algebras of rank $1$ over $\dbQ_p$: $\sl_2(\dbQ_p)$ and $\sl_1(D)$. 
\end{proof}

We are now ready to prove Theorem~\ref{main:prop} whose statement is recalled below:

\begin{Theoremmainprop}
Let $p$ be a prime and $G$ a pro-$p$ group. Then $G$ has finite NCC if and only
if one of the following 3 mutually exclusive conditions holds:
\begin{itemize}
\item[(i)] $G$ is finite. 
\item[(ii)] $G$ is infinite procyclic or $p=2$ and $G$ is infinite prodihedral, that is, the pro-$2$ completion of the infinite dihedral group.
\item[(iii)] $G$ is isomorphic to an open torsion-free subgroup of $\PGL_1(D)$ where $D$ is the quaternion division algebra over $\dbQ_p$. 
\end{itemize}
\end{Theoremmainprop}

\begin{proof} We start with the `only if' direction. In view of Corollary~\ref{cor:NCCLA}, it suffices to prove the following:
\begin{itemize}
\item[(1)] If $G$ has finite NCC and $L(G)\cong \dbQ_p$, then either $G\cong \dbZ_p$ or $p=2$ and $G$ is infinite prodihedral.
\item[(2)] If $L(G)\cong \sl_2(\dbQ_p)$, then $G$ has infinite NCC.
\item[(3)] If $L(G)\cong \sl_1(D)$, then $G$ is an open subgroup of $\PGL_1(D)$.
\item[(4)] If $G$ is an open pro-$p$ subgroup of $\PGL_1(D)$ with non-trivial torsion, then $G$ has infinite NCC.
\end{itemize}

(1) In this case $G$ must be virtually $\dbZ_p$. Let $Z$ be an open normal subgroup of $G$
isomorphic to $\dbZ_p$, and let $\phi:G\to\Aut(Z)$ be the map induced by conjugation. Since $Z$ is abelian, $\phi$ is not injective. 
Since $G$ has finite NCC, it is just-infinite, so $\Im\phi$ must be finite; in fact, a finite $p$-group. It is clear that $\Aut(\dbZ_p)\cong \dbZ_p^{\times}$, and it is well known  that $\dbZ_p^{\times}\cong\dbZ_p\times \dbZ/(p-1)\dbZ$ for $p>2$ and $\dbZ_2^{\times}\cong\dbZ_2\times \dbZ/2\dbZ$ 
(see, e.g. \cite[Corollary~5.8.2]{Go}). Thus, if $p>2$, then $\Aut(Z)$ has no non-trivial finite $p$-subgroups, so $\phi$ must be trivial, and if $p=2$, then $\Aut(Z)$ has a unique non-trivial finite subgroup which has order $2$. It follows that
$C_G(Z)$, the centralizer of $Z$ in $G$ (which coincides with $\Ker\phi$) equals the entire $G$ if $p>2$ and has index at most $2$ in
$G$ if $p=2$.
\skv

If $C_G(Z)$ contains a (non-trivial) torsion element $g$, then $\la g\ra\times Z$ is an open subgroup of $G$ 
which is not just-infinite and hence has infinite NCC, contrary to Lemma~\ref{lem:finiteindex}. Thus, $C_G(Z)$ is torsion-free.
A well-known theorem of Serre~\cite{Ser} asserts that a finitely generated pro-$p$ group which is virtually free and torsion-free must be free. Thus, $C_G(Z)\cong \dbZ_p$. Recall that $C_G(Z)=G$ if $p>2$, so we are done in the case. If $p=2$, we know that $G$ contains a subgroup
$U$ of index $\leq 2$ isomorphic to $\dbZ_2$. Thus, $G\cong\dbZ_2$ as well or $G$ contains a torsion element (necessarily of order $2$)
which acts on $U$ by inversion, in which case $G$ is infinite prodihedral.  

\skv
(2) Suppose now that $L(G)\cong \sl_2(\dbQ_p)$, so that $G$ contains an open subgroup of $\SL_2(\dbZ_p)$. 
Recall that an element $g$ of $\GL_n(F)$ for some field $F$ is called {\it unipotent} if
all if its eigenvalues are equal to $1$. After replacing $G$ by an open subgroup, we can assume
that 
\begin{itemize}
\item[(*)] all elements of $G$ are either diagonalizable or unipotent. 
\end{itemize}
Indeed, any element of $\SL_2(\dbZ_p)$ which is neither diagonalizable nor unipotent must have eigenvalue $-1$ with multiplicity $2$.
If $p\neq 2$, such an element cannot lie inside any pro-$p$ subgroup, and if $p=2$, such an element lies outside $\SL_2^2(\dbZ_2)$.

Now assume that $G$ has finite NCC, and apply Lemma~\ref{lem:NCCrest} where $A$ is the set of non-trivial unipotent elements in $G$, $B=\{1\}$ and $C=G\setminus (A\sqcup B)$. Condition (*) implies that $C$ is precisely the set
of non-trivial diagonalizable elements whence the hypotheses of Lemma~\ref{lem:NCCrest} are clearly satisfied. Note that $A\neq \{1\}$ as 
$G$ contains non-trivial elements of the form $\begin{pmatrix} 1&\lam \\ 0&1\end{pmatrix}$, so
by Lemma~\ref{lem:NCCrest}, $A\cup\{1\}$ has non-empty interior and hence the same is true for the set of all
unipotent elements in $\SL_2(\dbQ_p)$.  It is easy to see directly that the latter is false (e.g. using the fact that
the unipotent elements in $\SL_2(\dbQ_p)$ are precisely $2\times 2$ matrices with determinant $1$ and trace $2$).

\skv
(3) Since $L(\PGL_1(D))\cong \sl_1(D)$, we can immediately deduce that $G$ is commensurable with $\PGL_1(D)$, but proving that
$G$ is a subgroup of $\PGL_1(D)$ requires more work. We first recall the notion of the commensurator of a profinite group.

\begin{Definition}\rm
Let $P$ be a profinite group. The {\it commensurator of $P$}, denoted $\Comm(P)$, is the group of equivalence classes
of isomorphisms $\phi:U\to V$ where $U$ and $V$ are open subgroups of $P$. Here two isomorphisms $\phi:U\to V$ and $\phi':U'\to V'$
are equivalent if they coincide on an open subgroup of $U\cap U'$.
\end{Definition}

For any profinite group $P$ the conjugation action of $P$ on itself induces a canonical homomorphism $P\to \Comm(P)$. On the other
hand, if $Q$ is another profinite group which is commensurable to $P$ (that is, $Q$ and $P$ have isomorphic open subgroups),
then $\Comm(Q)\cong \Comm(P)$, and thus we obtain a homomorphism $\phi:P\to\Comm(Q)$.
\skv

We proceed with the proof in case (3). Since $G$ is commensurable with $\PGL_1(D)$, as we just explained, there is a natural homomorphism $\phi:G\to \Comm(\PGL_1(D))$. The image of $\phi$ must be infinite (for otherwise, $G$ is virtually abelian, which is clearly a contradiction) and hence by Corollary~\ref{lem:ji}, the kernel of $\phi$ must be trivial, so $G$ embeds into $\Comm(\PGL_1(D))$.
\skv

We claim that $\Comm(\PGL_1(D))\cong \PGL_1(D)$. Indeed, by \cite[Theorem~3.12]{BEW}, if $H$ is any compact $p$-adic analytic group,
then $\Comm(H)$ is isomorphic to $\Aut_{\dbQ_p}(L(H))$, so $\Comm(\PGL_1(D))\cong \Aut_{\dbQ_p}(\sl_1(D))$. 
By \cite[Proposition~8.1]{JT}, $\Aut_{\dbQ_p}(\sl_1(D))$ is isomorphic to $\Aut_{\dbQ_p}(D)$, the group of automorphisms of $D$
considered as an associative $\dbQ_p$-algebra.\footnote{The result in \cite{JT} is stated only for $p=2$, but the proof works for all $p$.} 
Finally, $\Aut_{\dbQ_p}(D)\cong \PGL_1(D)$ by the Skolem-Noether theorem.
\skv

Thus, we proved that $G$ is isomorphic to a (closed) subgroup of $\PGL_1(D)$, and since $G$ is commensurable with $\PGL_1(D)$,
this subgroup must be open (e.g. since $\PGL_1(D)$ is compact $p$-adic analytic, so its closed non-open subgroups have strictly
smaller dimension). 
 
\skv
(4) Finally, suppose that $G$ is an open pro-$p$ subgroup of $\PGL_1(D)$ with non-trivial torsion. Assume that $G$ has finite NCC.
 By Lemma~\ref{cor:proptorsion} the set of torsion elements in $G$ has non-empty interior. 
Let $T$ be set of all torsion elements in $\PGL_1(D)$.
Since $\PGL_1(D)$ is virtually torsion-free, the orders of torsion elements are bounded, so
 there exists $k\in\dbN$ such that $T=\{g\in \PGL_1(D): g^k=1\}$. Let $\rho:D^{\times}\to \PGL_1(D)$ be the natural projection.
 Then $\rho^{-1}(T)$ also has non-empty interior. On the other hand, if we identify $D$ with $\dbQ_p^4$ (by choosing any basis),
then $\rho^{-1}(T)\cup\{0\}$ is a proper Zariski closed subset of $\dbQ_p^4$ and thus must have empty interior, a contradiction.

\skv
This concludes the proof of the `only if' direction of Theorem~\ref{main:prop}. We now prove the `if' direction. It is clear that
the groups in families (i) and (ii) have finite NCC, so we only need to explain why open torsion-free pro-$p$ subgroups of $\PGL_1(D)$
have finite NCC. This fact was essentially known prior to this paper. It may have been indirectly observed by many mathematicians, but the earliest reference in the literature we are aware of is a paper of Jaikin-Zapirain~\cite{Ja}. 

Let us say that a group has {\it finite NAC} if it can be covered by the conjugacy classes of finitely many abelian subgroups or, equivalently, has finitely many conjugacy classes of maximal abelian subgroups.
Similarly to NCC, finiteness of NAC is preserved by passing to open subgroups.
The proof of Theorem~1.3 in \cite{Ja} shows that for any $p$-adic field $F$ and any finite-dimensional central division algebra $D$ over $F$,
the group  $\PGL_1(D)$ has finite NAC. It remains to show that if $F=\dbQ_p$, $\deg(D)=2$ and $G$ is an open torsion-free pro-$p$ subgroup of
$\PGL_1(D)$, then any maximal abelian subgroup of $G$ is procyclic. 

Since maximal abelian subgroups must be closed, it suffices to show
that any closed torsion-free abelian pro-$p$ subgroup of $\PGL_1(D)$ is procyclic. Any such group $A$ is also finitely generated
and hence isomorphic to $\dbZ_p^k$ for some $k\in\dbZ_{\geq 0}$. But then the Lie algebra $L(A)$ is abelian of dimension $k$, and $\sl_1(D)$ has no abelian subalgebras of dimension $>1$. Thus, $k\leq 1$, so $A$ is procyclic, as desired. 
\end{proof}
\skv
\begin{Remark} \rm The central problem investigated in \cite{Ja} is the following: given a pro-$p$ group $G$,
how fast/slow can the number of conjugacy classes of finite quotients $G/N$ grow relative to the size of $G/N$? Finiteness of NAC for the groups of the form $\PGL_1(D)$ was used in \cite{Ja} to show that for every $\eps>0$ there is a finitely generated pro-$p$ group $G$ 
such that the number of conjugacy classes of $G/N$ is at most $|G/N|^{\eps}$ whenever $|G/N|$ is sufficiently large.
\skv
Finiteness of NAC for the groups $\PGL_1(D)$, with $D$ as in Theorem~\ref{main:prop}, was also established
by B\"oge, Jarden and Lubotzky in \cite{BJL} using the same argument as in \cite{Ja}, but in a very different context. 
In the terminology of \cite{BJL}, a profinite group $G$ is called {\it sliceable} if there exist finitely many closed subgroups of infinite index $H_1,\ldots, H_k$ whose conjugacy classes cover $G$. \cite[Theorem~D]{BJL} asserts that the groups of the form $\PGL_1(D)$
are sliceable (but the proof shows they actually have finite NAC). The notion of a sliceable group was introduced in \cite{BJL} in connection with the number-theoretic problem on the existence of Kronecker field towers. It would be interesting to find any number-theoretic questions more directly related to Theorem~\ref{main:prop}.
\end{Remark}

\section{Profinite groups with finite NCC}

In this section we complete the proof of Theorem~\ref{thm:mainprofinite} by establishing the following result.

\begin{Theorem}
\label{thm:step3}
Let $G$ be a pro-$\Sol$ group with finite NCC, and suppose that $G^{(i)}$
is pro-$\Nil$ for some $i$. Then $G$ is virtually pro-$\Nil$.
\end{Theorem}

Theorem~\ref{thm:step3} is a fairly easy consequence of the following proposition:

\begin{Proposition}
\label{lem:metabelian}
Let $G$ be a metabelian profinite group with finite NCC, and let $A$ be an abelian closed normal subgroup of $G$ such that
$G/A$ is also abelian. The following hold:
\begin{itemize}
\item[(a)] $G$ has an open abelian subgroup containing $A$;
\item[(b)] $G$ is virtually procyclic.
\end{itemize}
\end{Proposition}

We will first prove Theorem~\ref{thm:step3} 
assuming Proposition~\ref{lem:metabelian} and then prove Proposition~\ref{lem:metabelian}.

\begin{proof}[Proof of Theorem~\ref{thm:step3}]
Let us consider the set of all pairs $(H,k)$ where $H$ is an open subgroup of $G$ and $k\in\dbZ_{\geq 0}$ is such that $H^{(k)}$ is pro-$\Nil$ (by hypotheses this set is non-empty). Among all such pairs $(H,k)$ choose one where $k$ is minimal. 
Theorem~\ref{thm:step3} is equivalent to the assertion that $k=0$.
\vskip .1cm

First we assume that $k\geq 2$ and consider the metabelian group $Q=H/H^{(2)}$. Since $Q$ has finite NCC (as $H$ does), it is virtually procyclic by Proposition~\ref{lem:metabelian}(b). Thus $H$ has an open subgroup $M$ whose image in $Q$ is procyclic and in particular abelian. Then $[M,M]\subseteq H^{(2)}$, whence $M^{(k-1)}= [M,M]^{(k-2)}\subseteq (H^{(2)})^{(k-2)}=H^{(k)}$, and so $M^{(k-1)}$ is pro-$\Nil$. Since $M$ is open in $H$ and hence in $G$,
this contradicts minimality of $k$. Thus we proved that $k\leq 1$.
\vskip .1cm

Since $k\leq 1$, the group $K=[H,H]$ is pro-$\Nil$. We will now use this fact to prove directly that $H$ (and hence $G$) is virtually pro-$\Nil$. By Proposition~\ref{lem:metabelian}(a), $Q=H/H^{(2)}$ has an open abelian subgroup $V$ containing $[H,H]/H^{(2)}$. If $U$ is the preimage of $V$ in $H$, then
$U$ is an open subgroup of $H$ containing $K=[H,H]$ such that $[U,U]\subseteq H^{(2)}$, so in particular $[U,K]\subseteq H^{(2)}=[K,K]$. 
Then $\gamma_3(U)=[U,[U,U]]\subseteq [U,K]\subseteq [K,K]$.

A well-known theorem of P. Hall asserts that if $X$ is a group which has a normal nilpotent subgroup $Y$ such that $X/[Y,Y]$ is nilpotent, then $X$ itself is nilpotent (see, e.g. \cite[5.2.10]{Ro}). It is straightforward to extend this theorem to pro-$\Nil$ groups. We know that
$K$ is pro-$\Nil$, and we just showed that $U/[K,K]$ is nilpotent of class $\leq 2$. Hence by Hall's theorem $U$ is pro-$\Nil$.
Since $U$ is open in $G$, the proof is complete. 
\end{proof}

We now turn to the proof of Proposition~\ref{lem:metabelian}. The proof presented below uses several ideas suggested by the referee and is much shorter and more conceptual than our original proof. 
We will need the following well-known result:

\begin{Lemma}[Proposition~5.5 in \cite{DDMS}]
\label{lem:AutFrattini} Let $G$ be a finitely generated pro-$p$ group, and let $\Aut(G,\Phi(G))$ be the kernel of the natural
map $\Aut(G)\to\Aut(G/\Phi(G))$. Then $\Aut(G,\Phi(G))$ is a pro-$p$ group and therefore $\Aut(G)$ is virtually pro-$p$.
\end{Lemma}

\begin{proof}[Proof of Proposition~\ref{lem:metabelian}]
We start by deducing (b) from (a). Abelian profinite groups with finite NCC are virtually procyclic -- this is not hard to prove directly, but also follows  from
 the classification of pro-$\Nil$ groups with finite NCC (Corollary~\ref{main:pronilpotent}) which is already completed at this stage.
Hence virtually abelian profinite groups with finite NCC are also virtually procyclic, so part (b) of Proposition~\ref{lem:metabelian}
indeed follows from part (a).
 
\skv
(a) First note that $G/A$ is an abelian profinite group with finite NCC and hence is virtually procyclic. Replacing $G$ by an open subgroup
containing $A$, we can assume from now on that
$G/A$ itself is procyclic. Thus, we can write $G=AD$ for some procyclic subgroup $D$.
Since $A$ and $D$ are abelian, they are direct product of Sylow pro-$p$ subgroups: $A=\prod A_p$ and $D=\prod D_p$
where $p$ ranges over all primes.

The conjugation action of $G$ induces maps $\pi:G\to\Aut(A)$ as well as $\pi_p:G\to\Aut(A_p)$ for each prime $p$.
Proposition~\ref{lem:metabelian}(a) asserts exactly that $\pi(G)$ is finite. 
Also  note that $\pi(G)=\pi(D)$ and $\pi_p(G)=\pi_p(D)$. 

We proceed with a few more auxiliary results.

\begin{Lemma}
\label{lem:dpfinite}
The minimal number of generators $d(A_p)$ is finite for all $p$.
\end{Lemma}
\begin{proof}
Fix a prime $p$, and let $\{G_i\}_{i=1}^{\infty}$ be any descending chain of open normal subgroups of $G$ which form a base of neighborhoods of identity. Let $V_p=A_p/A_p^p$, and let $\{V_{p,i}\}_{i\in\dbN}$ be the filtration of $V_p$ induced by $\{G_i\}$.
Since $G_i$ are open and normal in $G$, the subspaces $V_{p,i}$ are $G$-invariant and have finite codimension in $V_p$. 

Suppose now that $d_p=\infty$. Then $V_p$ is infinite and thus we can find an infinite sequence $i_1<i_2<\cdots$ such that the subspaces $V_{p,i_k}$ are all distinct. Choose
$v_k\in V_{p,i_k}\setminus V_{p,i_{k+1}}$. Then the subspaces $\dbF_p v_i$ and $\dbF_p v_j$ cannot be in the same $G$-orbit for $i\neq j$, so $G$ acts on the set of $1$-dimensional subspaces of $V_p$ with infinitely many orbits. On the other hand, the number of such orbits is exactly $\CC(V_p,G)$, and by
Lemma~\ref{CC:subgroup}(i)(ii) $\CC(V_p,G)\leq \NCC(G)$. Since $\NCC(G)<\infty$, we reached a contradiction. 
\end{proof}

\begin{Lemma}
\label{lem:pDp}
If $D_p$ is infinite, then $A_p$ is trivial.
\end{Lemma}
\begin{proof} Suppose this is false for some $p$. Then there exists an epimorphism $\rho_p:G\to \dbZ_p$ with 
$|\Ker\rho_p|$ (considered as a supernatural number) divisible by $p$. Since $G$ has finite NCC, this contradicts
Lemma~\ref{lem:ji2}.
\end{proof}

\begin{Lemma}
\label{lem:Vpfinite}
For each $p$ the group $\pi_p(D)$ is finite.
\end{Lemma}
\begin{proof} Let $D_p'=\prod_{q\neq p} D_q$, so that $D=D_p\times D_p'$. We already know that $\pi_p(D_p)$ is finite by 
Lemma~\ref{lem:pDp}, so it suffices to show that $\pi_p(D'_p)$ is finite.

Since $d(A_p)$ is finite by Lemma~\ref{lem:dpfinite}, the group $\Aut(A_p)$ is virtually pro-$p$ by Lemma~\ref{lem:AutFrattini}.
On the other hand, $|D_p'|$ is coprime to $p$, so its image $\pi_p(D'_p)$ has trivial intersection 
with any pro-$p$ subgroup of $\Aut(A_p)$. Hence $\pi_p(D'_p)$ has trivial intersection with an open subgroup
of $\Aut(A_p)$ and is therefore finite, as desired.
\end{proof}

Recall that for a group $\Phi$ acting by automorphisms on a profinite group $H$ we denote
by $CC(H,\Phi)$ the smallest number of procyclic subgroups of $H$ whose $\Phi$-orbits cover $H$. Equivalently,
$CC(H,\Phi)$ is the number of $\Phi$-orbits of maximal procyclic subgroups of $H$.

\begin{Lemma}
\label{lem:CC1}
There are only finitely many $p$ for which $CC(A_p,D)>1$.
\end{Lemma}
\begin{proof}
Supoose that $CC(A_p,D)>1$ for infinitely many $p$. Since $CC(A_p,D)=CC(A_p,G)=CC(A_p,\pi_p(G))$, Lemma~\ref{lem:product}
implies that $CC(A,\prod\limits_p \pi_p(G))=\infty$. Since $\pi(G)$ is a subgroup of $\prod\limits_p \pi_p(G)$,
we have $CC(A,G)=CC(A,\pi(G))\geq CC(A,\prod\limits_p \pi_p(G))$, and finally $NCC(G)\geq CC(A,G)$ by
Lemma~\ref{CC:subgroup}(i). Thus, $NCC(G)=\infty$, a contradiction.
\end{proof}

Since $\pi_p(G)$ is finite for each $p$, we can replace $G$ by $G/\prod\limits_{p\in F} A_p$ for any finite set of primes $F$
without affecting whether $\pi(G)$ is finite or not. In view of Lemma~\ref{lem:CC1}, after doing so we can assume
the following:

\begin{itemize}
\item[(*)] For each prime $p$ we have $CC(A_p,D)=1$, that is, any two maximal procyclic subgroups
of $A_p$ are conjugate by an element of $D$.
\end{itemize}

We proceed with the proof of Proposition~\ref{lem:metabelian}, now assuming (*). For each prime $p$
choose a maximal procyclic subgroup $C_p$ of $A_p$. Note that $C_p$ and $A_p$ have the same centralizer in $D$.
This follows from (*) and the fact that $D$ is abelian which implies that all $D$-conjugates of $C_p$
have the same centralizer.

\skv
Now let $P$ be the set of all primes $p$ such that
$A_p$ is non-trivial. If $P$ is finite, $\pi(G)$ is finite by Lemma~\ref{lem:Vpfinite}, so assume that
$P$ is infinite and enumerate its elements arbitrarily: $P=\{p_1,p_2,\ldots\}$.
For each $n\in\dbN$ let $P_n=\{p_1,p_2,\ldots, p_n\}$. Let
$C(n)=\prod\limits_{i=1}^n C_{p_i}$, let $B(n)$ be the centralizer of $C(n)$ in $D$ and 
$D(n)=\prod\limits_{p\not\in P_n}D_p\cap B(n)$ (note that the product is over all primes lying outside of $P_n$,
not just the ones in $P\setminus P_n$). The groups $\prod\limits_{p\not\in P_n}D_p$ and $B(n)$
are both open in $D$. The former holds since 
for any $p\in P_n$ the group $A_p$ is non-trivial and hence $D_p$ is finite by Lemma~\ref{lem:pDp}.
And $B(n)$ is open in $D$ by Lemma~\ref{lem:Vpfinite}. Hence $D(n)$ is also open in $D$.
\skv
Let $E(n)=C(n) D(n)$. By construction $C(n)$ and $D(n)$ are both procyclic and have coprime orders;
moreover $D(n)$ centralizes $C(n)$, so $E(n)$ is procyclic for each $n$. Since $G$ has finite NCC, it has a procyclic subgroup $L$
which contains some conjugate of $E(n)$ for arbitrarily large $n$. In particular, $L$ must contain
some conjugate $C'_{p_i}$ of $C_{p_i}$ for each $i$. On the other hand, $L$ contains a conjugate
of $D(m)$ for some $m$, and replacing $L$ by a conjugate, we can assume that $L$ contains $D(m)$.

We already observed that $C_{p_i}$ (and hence
also $C'_{p_i}$) has the same centralizer in $D$ as $A_{p_i}$. Hence the centralizer of $L$
in $D$ is contained in the centralizer of $A=\prod\limits_{i=1}^n A_{p_i}$. On the other hand, $L$ contains
$D(m)$ (which commutes with $D$ as $D$ is abelian), so $D(m)$ centralizes $A$. Thus $AD(m)$ is an abelian group containing $A$,
and it is open in $G$ since $D(m)$ is open in $D$. 
\end{proof}

\section{Connections with topology}
\label{sec:BVC}
The following terminology was introduced in \cite{vPW2}:
\begin{Definition}\rm A discrete group $G$ has {\it property (bCyc)} if $G$ has finite NCC.  
\end{Definition}
In this section we will introduce two variations of property (bCyc) denoted (bVC) and (BVC)
and discuss how they are related to (bCyc) and to each other. We will then explain how properties (bCyc)
and (BVC) naturally arise in the study of certain classifying spaces for families of subgroups. 

We start with a very general definition.

\begin{Definition}\rm  Let $G$ be a group and let $\mathcal F$ be a class of groups closed under isomorphisms and subgroups.
We will say that
\begin{itemize}
\item[(i)] $G$ has {\it property (b$\mathcal F$)} if there exist finitely many subgroups of $G$ which lie in $\calF$ and whose conjugacy classes cover $G$;
\item[(ii)] $G$ has {\it property (B$\mathcal F$)} if there exist finitely many subgroups $H_1,\ldots, H_k$ of $G$ which lie in $\calF$ and such that
every subgroup of $G$ lying in $\mathcal F$ is conjugate to a subgroup of $H_i$ for some $i$.
\end{itemize}
\end{Definition}
Below we will denote the classes of cyclic and virtually cyclic groups by $Cyc$ and $VC$, respectively. Clearly, properties
$(BCyc)$ and $(bCyc)$ are equivalent to each other and hold if and only if the group has finite NCC. The notation (BVC) was introduced in \cite{GW}, while properties (bCyc) and (bVC) were formally introduced in \cite{vPW2} (the notation for (bVC) in \cite{vPW} is (bVCyc)). 

\skv
 The following observation is immediate from definitions. 

\begin{Observation} 
\label{obs:BVC}
The following hold:
\begin{itemize}
\item[(a)] If $\mathcal F$ contains all cyclic groups, then (B$\mathcal F$) implies (b$\mathcal F$).
\item[(b)] If $\mathcal F_1\subseteq \mathcal F_2$, then (b$\mathcal F_1$) implies (b$\mathcal F_2$).
\item[(c)] If $\mathcal F$ is closed under quotients, then any quotient of a group with (b$\calF$) has (b$\calF$).
\end{itemize}
\end{Observation}

Thus either of the properties (BVC) and (BCyc)=(bCyc) implies (bVC). Unlike (bVC), property (BVC) is not inherited by quotients, and (bCyc) does not imply (BVC) (see Corollary~4.22 and Example~1.12 in \cite{vPW2}). There are plenty of groups which have
(BVC), but not (bCyc), e.g. any virtually cyclic group which is not finite, cyclic or infinite dihedral. However, discrete torsion-free groups with (BVC) have (bCyc) since a torsion-free virtually cyclic group must be cyclic.
The latter holds, for instance, since any infinite virtually cyclic group $V$ has a unique maximal finite normal
subgroup $N$ such that $V/N$ is infinite cyclic or infinite dihedral \cite[Lemma~4.1]{Wa}.
Further, residually finite groups with (BVC) are not far from having (bCyc):

\begin{Lemma}[see Lemma~5.0.2 in \cite{vP}] 
\label{lem:BVCfiniteindex}
Let $G$ be a discrete residually finite group with (BVC). Then some finite index subgroup $H$ of $G$ has (bCyc).
\end{Lemma}

Property (B$\mathcal F$) naturally arises in the study of the classifying space $\calE_{\calF}(G)$ defined as follows:

\begin{Definition}\rm Let $G$ be a discrete group and let $\calF$ be as above. A classifying space $\calE_{\calF}(G)$ is a $G$-$CW$ complex (that is, a $CW$ complex with a cellular action of $G$) such that for every subgroup $H$ of $G$, the $H$-fixed point space $\calE_{\calF}(G)^H$ is empty if $H\not\in\calF$ and
contractible (in particular, non-empty) if $H\in \calF$.
\end{Definition}

It is known that $\calE_{\calF}(G)$ is unique up to $G$-homotopy. 
\skv

A $G$-$CW$ complex is said to be {\it finite type} if it has finitely many $G$-orbits of cells in each dimension and {\it finite} if it
is of finite type and finite-dimensional. Juan-Pineda and Leary~\cite[Conjecture~1]{JPL} conjectured that a classifying space $\calE_{VC}(G)$
cannot be finite unless $G$ is virtually cyclic. A similar question of L\"uck, Reich, Rognes and Varisco~\cite[Question~4.9]{LRRV} asks
whether $\calE_{Cyc}(G)$ cannot be of finite type unless $G$ is finite, cyclic or dihedral. 

The following result establishes the basic relation between property (B$\calF$) for $G$ and the classifying space
$\calE_{\calF}(G)$:
\begin{Claim}
\label{claim:topconnection}
$G$ admits $\calE_{\calF}(G)$ with finitely many $0$-cells if and only if $G$ has (B$\calF$). 
\end{Claim}
Claim~\ref{claim:topconnection} in the case $\calF=VC$ is Lemma~1.3 in \cite{vPW}. The proof in the general case is identical.

\begin{Corollary} 
\label{cor:top}
Let $G$ be a residually finite group. Then
\begin{itemize}
\item[(a)] \cite[Conjecture~1]{JPL} holds for $G$ and 
\item[(b)] \cite[Question~4.9]{LRRV} has positive answer for $G$.
\end{itemize}
\end{Corollary}
\begin{proof} Suppose that $\calE_{Cyc}(G)$ has finite type. Then by Claim~\ref{claim:topconnection} $G$ has (BCyc)=(bCyc),
so (b) follows directly from Theorem~\ref{main:discrete}. To prove (a) we use the same argument in conjunction with Lemma~\ref{lem:BVCfiniteindex}.
\end{proof}

\end{document}